\documentclass[10pt, reqno]{amsart}
\usepackage{graphicx, amssymb, amsmath, amsthm}
\numberwithin{equation}{section}

\usepackage{color}

\usepackage[backref=page]{hyperref}  
\usepackage{amscd}
\usepackage{cite}
\usepackage{graphicx}
\usepackage{float}

\usepackage{amscd}
\def\pa{\partial}

\let\Re=\undefined\DeclareMathOperator*{\Re}{Re}

\newcommand{\R}{\mathbb{R}}

\newcommand{\Z}{\mathbb{Z}}

\newcommand{\LL}{\mathcal{L}}

\newtheorem{theorem}{Theorem}[section]

\newtheorem{lemma}[theorem]{Lemma}

\newtheorem{conjecture}[theorem]{Conjecture}
\newtheorem{proposition}[theorem]{Proposition}

\theoremstyle{definition}

\newtheorem{remark}[theorem]{Remark}

\makeatletter
\newcommand{\Extend}[5]{\ext@arrow0099{\arrowfill@#1#2#3}{#4}{#5}}
\makeatother

\begin{document}
\title[Linear restriction estimates]{ Restriction estimates in a conical singular space: wave equation }

\author{Xiaofen Gao}
\address{Department of Mathematics, Beijing Institute of Technology, Beijing 100081;}
\email{3120185710@bit.edu.cn}

\author{Junyong Zhang}
\address{Department of Mathematics, Beijing Institute of Technology, Beijing 100081; Department of Mathematics, Cardiff University, UK}
\email{zhang\_junyong@bit.edu.cn; ZhangJ107@cardiff.ac.uk}

\author{Jiqiang Zheng}
\address{Institute of Applied Physics and Computational Mathematics, Beijing 100088}
\email{zheng\_jiqiang@iapcm.ac.cn; zhengjiqiang@gmail.com}

\begin{abstract}

 We study the restriction estimates in a class of conical singular space $X=C(Y)=(0,\infty)_r\times Y$ with the metric $g=\mathrm{d}r^2+r^2h$, where the cross section $Y$ is a compact $(n-1)$-dimensional closed Riemannian manifold $(Y,h)$. Let $\Delta_g$ be the Friedrich extension positive Laplacian on $X$, and consider the operator $\mathcal{L}_V=\Delta_g+V$ with $V=V_0r^{-2}$, where $V_0(\theta)\in\mathcal{C}^\infty(Y)$ is a real function such that the operator $\Delta_h+V_0+(n-2)^2/4$ is positive. In the present paper, we  prove a type of modified restriction estimates for the solutions of wave equation associated with $\mathcal{L}_V$.  The smallest positive eigenvalue of the operator $\Delta_h+V_0+(n-2)^2/4$ plays an important role in the result.

As an application, for independent of interests, we prove local energy estimates and Keel-Smith-Sogge estimates for the wave equation in this setting.

\end{abstract}

 \maketitle

\begin{center}
 \begin{minipage}{100mm}

 \end{minipage}
 \end{center}

 \tableofcontents 

\section{Introduction}

\noindent

Restriction estimate as one of the cores in harmonic analysis is originally proposed by Stein \cite{Stein79} for sets $S$ having non-vanishing curvature, including hyper-surfaces such as sphere, paraboloid and cone. In this paper, we focus on the wave equation whose characteristic set is a cone, and so we take the cone to illustrate the details of the restriction estimates.\vspace{0.2cm}


Let $S$ be a smooth compact nonempty subset of the cone $\{(\tau,\xi)\in\R\times\R^n:\tau=|\xi|\}$ with $n\geq2$. For any Schwartz function $F$ on $S$, the inverse space-time Fourier transform of the measure $F\mathrm{d}\sigma$ is given by
\begin{align}\label{F-vee}
(F\mathrm{d}\sigma)^\vee(t,x)&=\int_SF(\tau,\xi)e^{2\pi i(x\cdot\xi+t\tau)}\mathrm{d}\sigma(\xi)\\\nonumber
&=\int_{\R^n}F(|\xi|,\xi)e^{2\pi i(x\cdot\xi+t|\xi|)}\frac{\mathrm{d}\xi}{|\xi|}.
\end{align}
where the conical measure $\mathrm{d}\sigma$ is the pullback of the measure $\frac{\mathrm{d}\xi}{|\xi|}$ under the projection $(\tau,\xi)\mapsto\xi.$ The restriction problem is to seek the optimal range of $p$ and $q$ satisfying the adjoint restriction estimate
\begin{equation}\label{est:res}
\|(F\mathrm{d}\sigma)^\vee\|_{L^q_{t,x}(\R\times\R^n)}\leq C_{p,q,n,S}\|F\|_{L^p(S,\mathrm{d}\sigma)}.
\end{equation}
The two necessary conditions such that \eqref{est:res} holds are
\begin{equation}\label{condition}
q>\frac{2n}{n-1}\quad \text{and}\quad \frac{n+1}{q}\leq\frac{n-1}{p'},
\end{equation}
which come from the decay of $(\mathrm{d}\sigma)^\vee$  and Knapp example, see \cite{Stein,Tao}.

A famous conjecture is to claim that the two necessary conditions also are sufficient for \eqref{est:res}, see \cite{Stein,Tao}. More precisely,
\begin{conjecture}\label{conj}
The estimate ~\eqref{est:res}~ is true if and only if ~\eqref{condition}~ hold.
\end{conjecture}

This conjecture is a great challenge and has attracted many mathematicians' attention. This conjecture has been proved to hold true  by Barcelo \cite{Barcelo}  for $n=2$ and Wolff \cite{Wolff} for $n=3$. Very recently, by using the method of polynomial partitioning developed by Guth\cite{Guth1,Guth2}, Ou and Wang\cite{OW} solved  the cone restriction conjecture for $n=4$ and made some new progress for the conjecture in the higher dimensions. The conjecture is so challenging that it remains open when $n\geq 5$. For recent work, see \cite{De,Tao} for more details on process of restriction estimate. \vspace{0.2cm}

It is known that the restriction problem on cone is closely related to the wave equation, e.g. see Tao \cite{Tao}.
Let $u$ be the solution of the wave equation
\begin{align}\label{equ:w}
\begin{cases}
&(\partial^2_t-\Delta)u(t,x)=0\quad\quad (t,x)\in\R\times \R^n\\
&u(0,x)=0, u_t(0,x)=f(x),\quad   x\in \R^n,
\end{cases}
\end{align}
then the solution
\begin{equation}
u(t,x)=\frac{\sin(t\sqrt{-\Delta})}{\sqrt{-\Delta}}f=\sum_{\pm}\int_{\R^n}\hat{f}(\xi)e^{2\pi i(x\cdot\xi\pm t|\xi|)}\frac{\mathrm{d}\xi}{|\xi|}.
\end{equation}
Take $F=\hat{f}$ in \eqref{F-vee}, then the inverse space-time Fourier transform of the measure $F\mathrm{d}\sigma$ equals each half-wave
\begin{equation}
(F\mathrm{d}\sigma)^\vee(\pm t,x)=\int_{\R^n}\hat{f}(\xi)e^{2\pi i(x\cdot\xi\pm t|\xi|)}\frac{\mathrm{d}\xi}{|\xi|}.
\end{equation}
The restriction problem for the wave equation is to find the optimal range of $p$ and $q$ satisfying the estimate
\begin{equation}\label{est:res-w}
\|u(t,x)\|_{L^q_{t,x}(\R\times\R^n)}\leq C_{p,q,n}\|\hat{f}\|_{L^p_{1/|\xi|}(\R^n)}.
\end{equation}
Since the support of $\hat{f}$ may not supported in a compact set, that is the above $S$ is the whole cone instead of a compact subset of the cone,
the necessary conditions \eqref{condition} are strengthened to

\begin{equation}\label{condition'}
q>\frac{2n}{n-1}\quad \text{and}\quad \frac{n+1}{q}=\frac{n-1}{p'}.
\end{equation}

The version of Conjecture \ref{conj} for wave equation can be stated as

\begin{conjecture}\label{conj-w}
The estimate ~\eqref{est:res-w}~ is true if and only if \eqref{condition'} hold.
\end{conjecture}

The above problems and conjectures are proposed to be associated with constant coefficient operator in the Euclidean space.
It is natural to ask analogous problems for the wave equation in a curve space or when there is a potential term in the equation.
In particular $p=2$, the inequality \eqref{est:res-w} is known as the Strichartz type estimate for wave equation.
There has been a lot of interest in developing Strichartz estimates on manifolds or when there is a potential term in the equation, both for the Schr\"odinger and wave equations;
this is too vast and highly active field to summarize here, but we refer to a very small and incomplete sample of recent results \cite{BPSS, MT, RS, HTW, ST}.
The restriction theory on manifolds arises in the study of eigenfunctions and the spectral measure of the Laplacian, for example, see Sogge \cite{sogge} on compact manifold and Guillarmou-Hassell-Sikora\cite{GHS} on asymptotically conic manifold.\vspace{0.2cm}

However, for general manifolds and $p\neq 2$, there is little result and the restriction theory is less satisfactory. Due to the geometry of the space and the spectrum of the operator, the results for the variable coefficient operator may be very different from the constant coefficient operator. For instance, one can not expect all the results from the Euclidean theory to carry over to curved space (e.g. see \cite{MS}). In this paper, we aim to prove a modified adjoint restriction inequality of \eqref{est:res-w} for the solution of \eqref{equ:w} in a conical singular space
$X$.\vspace{0.2cm}

Before stating our main result, we set up our model. Our setting is the metric cone $X$ which is a simple conical singular space as studied in
\cite{CT1,CT2, HL, MSe}. The conical space $(X,g)$ is given by the product space $X=C(Y)=(0,\infty)_r\times Y$ and the metric $g=dr^2+r^2h$ where $(Y,h)$ is a $(n-1)$-dimensional closed Riemannian manifold. A simplest example of a metric cone is the Euclidean space $\R^n$ when cross section $Y=\mathbb{S}^{n-1}$ and $h=d\theta^2$ the standard round metric on sphere. We stress that $X$ is more general and different from the Euclidean space. The space $X$ has an isolated conic singularity at cone tip $r=0$ except in the special case of Euclidean space.   The space $X$ does not have rotation symmetry and possibly has conjugate points due to the generality of $(Y,h)$
which bring many difficulties in the study of Strichartz estimates in \cite{ZZ1,ZZ2}. \vspace{0.2cm}

In this paper, as following the program in \cite{HL, Zhang, ZZ1,ZZ2}, we consider the Schr\"odinger operator
\begin{equation}\label{oper}
\mathcal{L}_V=\Delta_g+V
\end{equation}
where $\Delta_g$ is the Friedrichs extension of positive Laplace-Beltrami from the domain $\mathcal{C}_c^\infty(X^\circ),$ compactly supported smooth functions on the interior of the metric cone, and the potential $V=V_0(\theta)r^{-2}$ with $V_0(\theta)\in\mathcal{C}^\infty(Y)$ being such that the operator $\Delta_h+V_0+(n-2)^2/4$ is a strictly positive operator on $L^2(Y)$. There are many works studied this operator from different viewpoints. For example, wave diffraction phenomenon has been extensively studied in \cite{CT1,CT2};
 Riesz transform and heat kernel has been considered in \cite{HL, Li1, Li2, Mooer}. In the study of the regularity of wave propagator, Li \cite{Li3} and  M\"uller-Seeger  \cite{MSe} proved the $L^p$ regularity estimate; the Strichartz estimates were proved by Blair-Ford-Marzuola \cite{Ford, BFM} on flat cone $C(\mathbb{S}^1_\rho)$ and then were generalized
 by the last two authors in \cite{ZZ1,ZZ2}.\vspace{0.2cm}

Now, we state our main results. First, we study the adjoint restriction estimate for the solution of the wave equation
\begin{align}\label{equ:Lw}
\begin{cases}
&(\partial^2_t+\LL_V)u(t,r,\theta)=0,\quad\quad (t, r,\theta)\in\R\times X,\\
&u(0,r,\theta)=0, u_t(0,r,\theta)=f(r,\theta),\quad   (r,\theta)\in X.
\end{cases}
\end{align}
 More precisely, we prove

\begin{theorem}\label{thm:main}
 Let $n\geq2$ and $X$ be a n-dimensional metric cone, and let $\LL_V=\Delta_g+V$ where $r^2V=:V_0\in\mathcal{C}^\infty(Y)$ such that $\Delta_h+V_0(\theta)+(n-2)^2/4$ is a strictly positive operator on $L^2(Y)$ and its smallest
eigenvalue is $\nu_0^2$ with $\nu_0>0$. Suppose
 $f$ to be any Schwartz function and $u$ to be the solution of \eqref{equ:Lw}.

 $\bullet$ If $\nu_0\geq \tfrac{n-2}2$, and $(q,p)$ satisfies \eqref{condition'},
 then there exists a constant $C$ only depending on $p,q,n$ and $X$ such that
\begin{equation}\label{est:restriction}
\|u(t,r,\theta)\|_{L^q_{t}(\R;L^q_{\mathrm{rad}}(L^2_\mathrm{sph}))}\leq C_{p,q,n,X}\|\rho^{-\frac{1}{p}}\hat{f}(\rho,\omega)\|_{L^p_{\mathrm{rad}}(L^2_\mathrm{sph})},
\end{equation}
where $L^q_{\mathrm{rad}}(L^2_\mathrm{sph})=L^q_{r^{n-1}dr}((0,\infty);L^{2}_\theta(Y))$ and $\hat{f}$ denotes the distorted Fourier transform defined in \eqref{Fourier} below.

 $\bullet$ If $0<\nu_0< \tfrac{n-2}2$, and $(q,p)$ satisfies \eqref{condition'} and
 \begin{equation}\label{add-cond}
  q<\frac{2n}{n-2-2\nu_0}
 \end{equation}
 then \eqref{est:restriction} holds true. The additional requirement \eqref{add-cond} is necessary.

\end{theorem}

\begin{remark} The admissible pair $(q,p)$ here is almost the same as stated in Conjecture \ref{conj-w} when $\nu_0\geq (n-2)/2$. However, from the additional necessary condition \eqref{add-cond}, the smallest eigenvalue of $\Delta_h+V_0(\theta)+(n-2)^2/4$ plays an important role.
\end{remark}

\begin{remark}\label{rem:stri} In particular $p=2$, from \eqref{est:restriction}, we obtain the Strichartz estimate
\begin{equation}\label{est:stri}
\|u(t,r,\theta)\|_{L^{\frac{2(n+1)}{n-1}}_{t}(\R;L^{\frac{2(n+1)}{n-1}}_{\mathrm{rad}}(L^2_\mathrm{sph}))}\leq C\|f\|_{\dot{H}^{-\frac12}(X)}
\end{equation}
which is weaker than the Strichartz estimate proved in \cite{ZZ2}. However, from \eqref{est:stri''} below, we can prove
\begin{equation}\label{est:stri'}
\|u(t,r,\theta)\|_{L^q_{t}(\R;L^{q}_{\mathrm{rad}}(L^2_\mathrm{sph}))}\leq C M^{\frac{n-1}2-\frac{n+1}q}\|f\|_{\dot{H}^{-\frac12}(X)}, \quad q>\tfrac{2n}{n-1}
\end{equation}
provided $\text{supp} (\hat{f})\subset\{\rho\sim M\}$, and one needs the restriction $q<\tfrac{2n}{n-2-2\nu_0}$ when $0<\nu_0<\tfrac{n-2}2.$
This includes some new Strichartz estimates since one can choose $q$ to be out of the admissible assumption in \cite{ZZ2}.

\end{remark}

\begin{remark}
Compared with \cite{Zhang}, in which the second author studied the restriction estimate for Schr\"odinger equation associated with $\mathcal{L}_V$,
here we remove the positive assumption on the potential and improve the loss of angular regularity.
\end{remark}

\begin{remark}
The conjugate points do not effect the estimates \eqref{est:restriction} due to the mixed space $L^q_{\mathrm{rad}}(L^2_\mathrm{sph})$.
It would be interesting to investigate the implicit influence of conjugate point by establishing \eqref{est:res-w}.
\end{remark}

We stress that the result does not solve the challenging Conjecture \ref{conj-w} in the conical singular space since we use a mixed space $L^q_{\mathrm{rad}}(L^2_\mathrm{sph})$.  The modified norm in \eqref{est:restriction} is motivated to simplify the Conjecture \ref{conj-w} from two aspects: the wavelet tubes overlap in the angular
direction and the parametrix of the wave propagator. \vspace{0.2cm}

The modified norm $L^q_{\mathrm{rad}}(L^2_\mathrm{sph}(\mathbb{S}^{n-1}))$ has been used in many famous harmonic analysis problems (such as Fourier
restriction estimates, local smoothing conjecture  etc.) on Euclidean space $\R^n$, we refer the reader to \cite{ Cor, CG, CRS, GS, Moc, Shao}. Miao and the last two authors \cite{MZZ, MZZ1} proved the restriction estimates for wave and Schr\"odinger equation when the initial data has additional angular regularity.
And later, C\'ordoba-Latorre \cite{CL1} revisited some classical conjectures including restriction estimate in harmonic analysis in the mixed space $L^q_{\mathrm{rad}}(L^2_\mathrm{sph})$.  \vspace{0.2cm}

In the conical singular space $X$,
the spacetime Fourier transform is no longer so useful as well as in Euclidean space, and so restriction theory is harder to be established;
however the spatial Fourier transform can be replaced by the spectral decomposition of the Laplacian,
some techniques used in that theory still do apply, for instance the $TT^*$-method which was used in \cite{GHS} to prove Stein-Tomas type restriction estimates and in \cite{HZ, ZZ1, ZZ2} to prove Strichartz estimates. The  $TT^*$-strategy is a key point in those papers to use the approximate microlocalized parametrix for the fundamental solution
 which is more complicated than the Euclidean's due to the possibility of appearing conjugate points in the space. But if one aims to establish
\eqref{est:res-w} when $p\neq 2$, the $TT^*$-method breaks down. Instead of using the microlocalized parametrix constructed in \cite{ZZ1, ZZ2},
we will use the method of Cheeger-Taylor \cite{CT1,CT2}, even  though the method leads to a loss of angular regularity.
The method of Cheeger-Taylor has been used by  M\"uller-Seeger \cite{MSe} to establish local smoothing estimates in the mixed spacetime $L^p_{\mathrm{rad}}(L^2_\mathrm{sph})$ estimates
for wave equation in this conical singular space.
\vspace{0.2cm}

Our strategy of proving Theorem \ref{thm:main} is to establish the localized estimates for Hankel transform by analyzing Bessel function and using stationary phase argument. As an application of the proof, for independent interest, we will prove local energy estimates and Keel-Smith-Sogge estimates in our setting.

\begin{theorem}\label{thm:KSS}
Let ${\bf R}>0$ be a fixed number and let $u(t,r,\theta)$ be the solution of \eqref{equ:Lw} with initial data $f\in\dot{H}^{-1}$. Then the following results hold:

$\bullet$ Local energy decay estimate:
\begin{equation}\label{est:locendec}
\sup_{{\bf R}>0}{\bf R}^{-1/2}\|u(t,r,\theta)\|_{L^2(\R; L^2((0,{\bf R}]\times Y))} \lesssim \|f\|_{\dot{H}^{-1}}
\end{equation}

$\bullet$ Keel-Smith-Sogge estimate: let $\beta>0$
\begin{equation}\label{kss}
\|\langle r\rangle^{-\beta} u(t,r,\theta)\|_{L^2([0,T];L^2(X))}\leq C_\beta(T)\|f\|_{\dot{H}^{-1}},
\end{equation}
where
$$
C_\beta(T)=C\times
\begin{cases}
T^{\frac12-\beta},\quad &\text{if}\;\quad 0\leq \beta< \frac12,\\
(\log((2+T))^{\frac12},\quad  &\text{if}\;\quad \beta=\frac12,\\
1 \quad  &\text{if}\;\quad \beta>\frac12,
\end{cases}
$$
where $C$ is an absolute constant independent of $T$.

$\bullet$ Local smoothing estimate: let $0\leq \beta <\tfrac12$
\begin{equation}\label{kss2}
\big\| |x|^{-\beta} u(t,r,\theta) \big\|_{L^2([0,T];L^2(X))}\lesssim T^{\frac12-\beta}\|f\|_{\dot{H}^{-1}}.
\end{equation}

\end{theorem}

\begin{remark}
If $\tfrac12<\beta<1+\nu_0$ with $\nu_0$ being given in Theorem \ref{thm:main}, a global-in-time local smoothing estimate \eqref{kss2} (that is, the constant is independent of $T$), has been proved by the last two authors in \cite{ZZ2}.
\end{remark}

\begin{remark} The Keel-Smith-Sogge (KSS) estimates were originally developed in \cite{KSS} to study the lifespan of solution of quasilinear wave equation.
We present the KSS estimates here for independent interests of studying the existence theory of the solution of nonlinear wave equation (e.g. Strauss conjecture and Glassey conjecture) in this setting.

\end{remark}

This paper is organized as follows: Section 2 gives some preliminaries including the spectral properties, Bessel function and Hankel transform.
In Section 3, we prove the key localized estimates of Hankel transform. The proof of Theorem \ref{thm:main} is presented in Section 4. Section 5 provides the proof of Theorem \ref{thm:KSS}.\vspace{0.2cm}

{\bf Acknowledgments:}\quad
The authors were supported by National Natural
Science Foundation of China (11771041, 11831004, 11901041,11671033) and H2020-MSCA-IF-2017(790623).
\vspace{0.2cm}




\section{Preliminaries}

In this section, we recall spectral and harmonic analysis results such as orthogonal decomposition of $L^2(Y)$, some basic properties about Hankel transform and Bessel function. In the end of this section, we introduce some notations.

\subsection{Spectral property of $\Delta_h+V_0(y)+(n-2)^2/4$} To study the operator $\LL_V$, we recall some spectral result of $\Delta_h+V_0(y)+(n-2)^2/4$, e.g. see  \cite{Wang, Zhang}.

Consider the operator in \eqref{oper}
\begin{equation}\label{oper'}
\LL_V=\Delta_g+\frac{V_0(\theta)}{r^2},
\end{equation}
on the metric cone $X=(0,\infty)_r\times Y.$ In coordinates $(r,\theta)\in\R_+\times Y$, $V_0(\theta)$ is a real continuous function and the metric $g$ takes the form
$$g=\mathrm{d}r^2+r^2h(\theta,\mathrm{d}\theta),$$ where  $h$ is the Riemannian metric  on $Y$ independent of $r$. Let $\Delta_h$ be the positive Laplace-Beltrami operator on $(Y,h)$ and  let $\nu_0^2$ be the smallest eigenvalue of the operator $\Delta_h+V_0(\theta)+(n-2)^2/4$, that is, for any $f\in L^2(Y)$, it holds \footnote{The assumption here
is weaker than the hypothesis in \cite{Zhang} where one needs $\Delta_h+V_0(\theta)\geq 0$.}
\begin{equation}\label{nu0}
\big\langle(\Delta_h+V_0(\theta)+(n-2)^2/4)f,f\big\rangle_{L^2(Y)}\geq \nu^2_0\|f\|^2_{L^2(Y)}.
\end{equation}
Let $\nu^2$ be one eigenvalue of the operator $\Delta_h+V_0(\theta)+(n-2)^2/4$ such that
\begin{equation}
(\Delta_h+V_0(\theta)+(n-2)^2/4) Y(\theta)=\nu^2 Y(\theta)
\end{equation}
where $Y(\theta)$ is an eigenfunction. Since $Y$ is a closed manifold, from the spectral theory, it is known that $\nu^2$ falls in a discrete
set, say $\{\nu_j^2\}_{j=0}^\infty$, and moreover $\nu^2_0<\nu^2_1<\cdots<\nu^2_j<\cdots \to \infty$.
Let $d(\nu_j)$ be the multiplicity of $\nu_j^2$ and let $\{Y_{\nu_j,\ell}(\theta)\}_{1\leq\ell\leq d(\nu_j)}$ be the corresponding eigenfunctions of $\Delta_h+V_0(\theta)+(n-2)^2/4$, that is
\begin{equation}\label{equ:eig}
\begin{split}
(\Delta_h+V_0(\theta)+(n-2)^2/4)Y_{\nu_j,\ell}(\theta)=\nu_j^2 Y_{\nu_j,\ell}(\theta),\\ \langle Y_{\nu_j,\ell},Y_{\nu_{j'},\ell'}\rangle_{L^2(Y)}=\delta_{j,j'}\delta_{\ell,\ell'}.
\end{split}
\end{equation}
where $\delta$ is the Kronecker delta function. In particular, when $Y=\mathbb{S}^{n-1}$ and $V_0=0$, $Y_{\nu_j,\ell}$ is spherical harmonics.
Define
\begin{equation}\label{Lam}
\Lambda_\infty=\big\{\nu_j>0: \nu_j^2\, \text{is the eigenvalue of }\, \Delta_h+V_0(\theta)+(n-2)^2/4\big\}_{j=0}^\infty.
\end{equation}
 From now  on, we drop the superscripts in $\nu_j$ for simple.
Define $$\mathcal{H}^\nu=\text{span}\{Y_{\nu,1},\cdots,Y_{\nu,d(\nu)}\},$$
then we have the orthogonal decomposition $$L^2(Y)=\bigoplus_{\nu\in\Lambda_\infty}\mathcal{H}^\nu.$$ Let $\pi_\nu$ denote the orthogonal projection:
$$\pi_\nu f=\sum_{\ell=1}^{d(\nu)}Y_{\nu,\ell}(\theta)\int_Yf(r,\omega)Y_{\nu,\ell}(\omega)\mathrm{d}\sigma_h,\quad f\in L^2(X),$$
where $\mathrm{d}\sigma_h$ is the measure on $Y$ under the metric $h$. For any $g\in L^2(X),$ we have the expansion formula
\begin{equation}\label{2.5}
g(r,\theta)=\sum_{\nu\in\Lambda_\infty}\pi_\nu g=\sum_{\nu\in\Lambda_\infty}\sum_{\ell=1}^{d(\nu)}a_{\nu,\ell}(r)Y_{\nu,\ell}(\theta)
\end{equation}
where $a_{\nu,\ell}(r)=\int_Yg(r,\theta)Y_{\nu,\ell}(\theta)\mathrm{d}\sigma_h.$ By orthogonality, it gives
\begin{equation}\label{orth}
\|g(r,\theta)\|^2_{L^2(Y)}=\sum_{\nu\in\Lambda_\infty}\sum_{\ell=1}^{d(\nu)}|a_{\nu,\ell}(r)|^2.
\end{equation}

\subsection{The Bessel Function and Hankel Transform}

For our purpose, we recall the Bessel function $J_\nu(r)$ of order $\nu$, which is defined by $$J_\nu(r)=\frac{(r/2)^\nu}{\Gamma(\nu+\frac{1}{2})\Gamma(1/2)}\int_{-1}^1e^{isr}(1-s^2)^{(2\nu-1)/2}\mathrm{d}s,$$
where $\nu>-\frac{1}{2}$ and $r>0.$ A simple computation gives the rough estimate
\begin{equation}\label{est:r}
|J_\nu(r)|\leq\frac{Cr^\nu}{2^\nu\Gamma(\nu+\frac{1}{2})\Gamma(\frac12)}\Big(1+\frac{1}{\nu+\frac12}\Big),
\end{equation}
where $C$ is an absolute constant independent of $r$ and $\nu$.

To investigate the behavior of asymptotic on $\nu$ and $r$, we recall Schl\"afli's integral representation \cite{Wolff} of the Bessel function: for $r\in\R^+$ and $\nu>-\tfrac12$
\begin{align}\label{SIR}
J_\nu(r)&=\frac{1}{2\pi}\int_{-\pi}^{\pi}e^{ir\sin\theta-i\nu\theta}\mathrm{d}\theta-\frac{\sin(\nu\pi)}{\pi}\int_0^\infty e^{-(r\sinh s+\nu s)}\mathrm{d}s \nonumber \\
&:=\tilde{J}_\nu(r)-E_\nu(r).
\end{align}
We remark that $E_\nu(r)=0$ when $\nu\in\mathbb{Z}^+.$ A simple computation gives that for $r>0$
\begin{equation}\label{est:E}
|E_\nu(r)|=\Big|\frac{\sin(\nu\pi)}{\pi}\int_0^\infty e^{-(r\sinh s+\nu s)}\mathrm{d}s\Big|\leq C(r+\nu)^{-1}.
\end{equation}
Next, we recall the properties of Bessel function $J_\nu(r)$ in \cite{SW, Watson}.
\begin{lemma}[Asymptotic of the Bessel function]\label{lem:bessel}
Assume $\nu\gg1.$ Let $J_\nu(r)$ be the Bessel function of order $\nu$ defined as above. Then there exist a large constant $C$ and a small constant $c$ independent of $\nu$ and $r$ such that:

$\bullet$ when $r\leq\frac{\nu}{2}$
\begin{equation}\label{2.12}
|J_\nu(r)|\leq Ce^{-c(\nu+r)};
\end{equation}

$\bullet$ when $\frac{\nu}{2}\leq r\leq2\nu$
\begin{equation}\label{2.13}
|J_\nu(r)|\leq C\nu^{-\frac{1}{3}}(\nu^{-\frac{1}{3}}|r-\nu|+1)^{-\frac{1}{4}};
\end{equation}

$\bullet$ when $r\geq2\nu$
\begin{equation}\label{2.14}
|J_\nu(r)|=r^{-\frac{1}{2}}\sum_{\pm}a_\pm(r,\nu)e^{\pm ir}+E(r,\nu),
\end{equation}
where $|a_\pm(r,\nu)|\leq C$ and $|E(r,\nu)|\leq Cr^{-1}.$

\end{lemma}
Let $f\in L^2(X),$ we define the Hankel transform of order $\nu$ by
\begin{equation}\label{Hankel}
(\mathcal{H}_\nu f)(\rho,\theta)=\int_0^\infty(r\rho)^{-\frac{n-2}{2}}J_\nu(r\rho)f(r,\theta)r^{n-1}\mathrm{d}r.
\end{equation}
In particular,  if the function $f$ is independent of $\theta$,
then
\begin{equation}\label{Hankel-r}
(\mathcal{H}_{\nu}f)(\rho)=\int_0^\infty(r\rho)^{-\frac{n-2}2}J_{\nu}(r\rho)f(r)r^{n-1}\mathrm{d}r.
\end{equation}
We have the following properties of the Hankel transform. We refer the readers to M.Taylor \cite[Chapter 9]{Taylor}, also see  \cite{BPSS,PSS}.
\begin{lemma}\label{lem:hankel}
Let $\mathcal{H}_\nu$ be the Hankel transform in \eqref{Hankel} and
\begin{equation}\label{LV-nu}
A_\nu:=-\partial^2_r-\frac{n-1}{r}\partial_r+\frac{\nu^2-(\frac{n-2}{2})^2}{r^2}.
\end{equation}
Then
\begin{item}
\item $\mathrm{(1)}$ $\mathcal{H}_\nu=\mathcal{H}^{-1}_\nu,$

\item $\mathrm{(2)}$ $\mathcal{H}_\nu$ is self-adjoint, i.e.$\quad \mathcal{H}_\nu=\mathcal{H}^{*}_\nu$,

\item $\mathrm{(3)}$ $\mathcal{H}_\nu$ is an $L^2$ isometry, i.e. $\|\mathcal{H}_\nu f\|_{L^2(X)}=\|f\|_{L^2(X)},$

\item $\mathrm{(4)}$ $\mathcal{H}_\nu(A_\nu f)(\rho,\theta)=\rho^2(\mathcal{H}_\nu f)(\rho,\theta),$ for $f\in L^2.$

\end{item}

\end{lemma}

\subsection{Distorted plan wave and distorted Fourier transform} In this subsection, we derive the plan wave associated with the operator $\LL_V$.
To this end, we need to find the eigenfunction $\phi(r,\theta;\rho,\omega)$ such that
\begin{equation}\label{equ:eigen}
\LL_V\phi(r,\theta;\rho,\omega)=\rho^2\phi(r,\theta;\rho,\omega).
\end{equation}
We claim that
\begin{equation}\label{plan-w}
\phi(r,\theta;\rho,\omega)=(r\rho)^{-\frac{n-2}2}\sum_{\nu\in\Lambda_\infty} J_{\nu}(r\rho)\sum_{\ell=1}^{d(\nu)}Y_{\nu,\ell}(\theta)\overline{Y_{\nu,\ell}(\omega)}
\end{equation}
where $J_\nu$ is the Bessel
function of order $\nu$ and $Y_{\nu,\ell}$ satisfies \eqref{equ:eig}.
To verify this claim,
we write $\LL_V$ in the coordinates $(r,\theta)$ as
\begin{equation}\label{LV-r}
\LL_V=-\partial^2_r-\frac{n-1}{r}\partial_r+\frac{1}{r^2}(\Delta_h+V_0(\theta)),
\end{equation}
if it acts on the function in each $\mathcal{H}^\nu$, then it equals to $A_\nu$  as in \eqref{LV-nu}
\begin{equation}
A_\nu:=-\partial^2_r-\frac{n-1}{r}\partial_r+\frac{\nu^2-(\frac{n-2}{2})^2}{r^2}.
\end{equation}
Therefore it suffices to verify that: for each $\nu$, let $F(r\rho)=(r\rho)^{-\frac{n-2}2}J_{\nu}(r\rho)$, one has
\begin{equation}\label{bess}
\rho^2 F''(r\rho)+\frac{(n-1)\rho}r F'(r\rho)+\Big[\rho^2-\frac{\nu^2-(n-2)^2/4}{r^2}\Big]F(r\rho)=0.
\end{equation}
Indeed, the Bessel function $J_\nu(\lambda)$ solves
\begin{equation}\label{bessfunc}
G''(\lambda)+\frac{1}\lambda
G'(\lambda)+\Big[1-\frac{\nu^2}{\lambda^2}\Big]G(\lambda)=0,
\end{equation}
let $\lambda=r\rho$,  then $F(\lambda)$ satisfies
\begin{equation*}
F''(\lambda)+\frac{n-1}\lambda F'(\lambda)+\Big[1-\frac{\nu^2-(n-2)^2/4}{\lambda^2}\Big]F(\lambda)=0
\end{equation*}
which implies \eqref{bess}. For $$f(r,\theta)=\sum\limits_{\nu\in\Lambda_\infty}\sum\limits_{\ell=1}^{d(\nu)}a_{\nu,\ell}(r)Y_{\nu,\ell}(\theta)\in L^2(X),$$ we define
the distorted Fourier transform
\begin{equation}\label{Fourier}
\begin{split}
\hat f(\rho,\omega)&=\int_0^\infty\int_Y f(r,\theta)\overline{\phi(r,\theta;\rho,\omega)} \,r^{n-1}\mathrm{d}r \mathrm{d}h\\&=\sum_{\nu\in\Lambda_\infty}\sum_{\ell=1}^{d(\nu)}
Y_{\nu,\ell}(\omega)\big(\mathcal{H}_{\nu}a_{\nu,\ell}\big)(\rho).
\end{split}
\end{equation}

\subsection{The representation of solution} Based on the above Hankel transform, we write out  the explicit expression of solution for wave equation
\eqref{equ:Lw}. Recall \eqref{LV-r} in coordinates $(r,\theta)$, then the solution $u(t,r,\theta)$
 satisfies that
\begin{equation}\label{equ:vt}
\begin{cases}
\pa_{tt}u-\pa_{rr}u-\frac{n-1}{r}\pa_ru+\frac1{r^2}\Delta_hu+\frac{V_0(\theta)}{r^2}u=0,\\
u(0,r,\theta)=0, \quad \partial_tu(0,r,\theta)=f(r,\theta).
\end{cases}
\end{equation}
We write Schwartz function $f(r,\theta)$ as
\begin{equation}\label{in-f}
f(r,\theta)=\sum_{\nu\in\Lambda_\infty}\sum_{\ell=1}^{d(\nu)}a_{\nu,\ell}(r)Y_{\nu,\ell}(\theta).
\end{equation}
With separation of variables, then we can write $u$ as a superposition
\begin{equation}\label{3.3}
u(t,r,\theta)=\sum_{\nu\in\Lambda_\infty}\sum_{\ell=1}^{d(\nu)}u_{\nu,\ell}(t,r)Y_{\nu,\ell}(\theta),
\end{equation}
where $u_{\nu,\ell}$ satisfies
\begin{equation}\label{equ:vkl}
\begin{cases}
\pa_{tt}u_{\nu,\ell}-\pa_{rr}u_{\nu,\ell}-\frac{n-1}r\pa_ru_{\nu,\ell}+\frac{\nu^2-(\frac{n-2}{2})^2}{r^2}u_{\nu,\ell}=0,\\
u_{\nu,\ell}(0,r)=0,\quad \partial_tu_{\nu,\ell}(0,r)=a_{\nu,\ell}(r)
\end{cases}
\end{equation}
for each $\nu\in\Lambda_\infty,\ell\in\mathbb{N}$ and ~$1\leq\ell\leq d(\nu)$. Recall $A_\nu$ in \eqref{LV-nu},
we consider
\begin{equation}\label{equ:vkln}
\begin{cases}
\pa_{tt}u_{\nu,\ell}+A_{\nu}u_{\nu,\ell}=0,\\
u_{\nu,\ell}(0,r)=0,\quad \partial_tu_{\nu,\ell}(0,r)=a_{\nu,\ell}(r).
\end{cases}
\end{equation}
Applying the Hankel transform to ~\eqref{equ:vkln}~, by Lemma \ref{lem:hankel}, we have
\begin{equation}\label{equ:vkln1}
\begin{cases}
\pa_{tt}\tilde{u}_{\nu,\ell}+\rho^2\tilde{u}_{\nu,\ell}=0,\\
\tilde{u}_{\nu,\ell}(0,\rho)=0,\quad \partial_t\tilde{u}_{\nu,\ell}(0,\rho)=b_{\nu,\ell}(\rho),
\end{cases}
\end{equation}
where
\begin{equation}\label{3.8}
\tilde{u}_{\nu,\ell}(t,\rho)=(\mathcal{H}_{\nu}
u_{\nu,\ell})(t,\rho),\quad
b_{\nu,\ell}(\rho)=(\mathcal{H}_{\nu}a_{\nu,\ell})(\rho).
\end{equation}
By solving this ODE and using the Hankel transform, we obtain
\begin{equation*}
u_{\nu,\ell}(t,r)=\mathcal{H}_{\nu}[\rho^{-1}\sin (t\rho) \, b_{\nu,\ell}(\rho)](r)
\end{equation*}
Therefore, by \eqref{3.3} and the definition of Hankel transform, we get
\begin{align}\label{sol}
 u(t,r,\theta)&=\sum_{\nu\in\Lambda_\infty}\sum_{\ell=1}^{d(\nu)}\mathcal{H}_{\nu}\big[\rho^{-1}\sin (t\rho) b_{\nu,\ell}(\rho)\big](r)Y_{\nu,\ell}(\theta)\nonumber\\
&=\sum_{\nu\in\Lambda_\infty}\sum_{\ell=1}^{d(\nu)}Y_{\nu,\ell}(\theta) \int_0^\infty(r\rho)^{-\frac{n-2}2}J_{\nu}(r\rho)\sin (t\rho) b_{\nu,\ell}(\rho)\rho^{n-2}\;d\rho .
\end{align}

We finally record the Van der Corput lemma for convenience.

\begin{lemma}\label{lem:VdC}
Let $\phi$ be a smooth real-valued function defined on an interval $[a,b]$ and let $|\phi^{(k)}(x)|\geq1$ for all $x\in[a,b].$ Then
\begin{equation}\label{VdC}
|\int_a^be^{i\lambda\phi(x)}\mathrm{d}x|\leq c_k\lambda^{-\frac{1}{k}}
\end{equation}
holds when:

$\bullet\quad k\geq2$ or

$\bullet\quad k=1$ and $\phi'(x)$ is monotonic.

The bound $c_k$ is independent of $\phi$ and $\lambda.$

\end{lemma}

\subsection{Notation}
We use $A\lesssim B$ to denote the statement that $A\leq CB$ for some large constant $C$ which may vary from line to line and depend on various parameters, and similarly, we employ $A\backsim B$ to state that $A\lesssim B\lesssim A.$ We also use $A\ll B$ to denote the statement $A\leq C^{-1}B.$ If a constant $C$ depends on a special parameter other than the above, we shall denote it explicitly by subscripts. For instance, $C_\epsilon$ should be understood as a positive constant not only depending on $p,q,n$ and $S,$ but also on $\epsilon.$ Throughout this paper, pairs of conjugate indices are written as $p,p',$ where $\frac{1}{p}+\frac{1}{p'}=1$ with $1\leq p\leq\infty.$ Let $R>0$ be two dyadic numbers, we define $S_R=[R/2,R]$.




\section{Localized estimates of Hankel transform}

In this section, we utilize the stationary-phase argument to prove the estimates for Hankel transform localized both in frequency and physical spaces.
These inequalities are key to prove main theorem in next section.

\begin{proposition}\label{LRE}

Let $\varphi\in\mathcal{C}_c^\infty(\R)$ be supported in $I:=[1,2]$ and let $R>0$ be a dyadic number and $S_R=[R/2,R]$.
Then
\begin{align}\label{L-2}
&\Big\|\Big(\sum_{\nu\in\Lambda_\infty}\sum_{\ell=1}^{d(\nu)}\Big|\mathcal{H}_{\nu}\big[\rho^{-1}e^{\pm it\rho}\varphi(\rho) b_{\nu,\ell}(\rho)\big](r)\Big|^2\Big)^{1/2}\Big\|_{L^2_tL^2_{r^{n-1}\mathrm{d}r}(\R\times S_R)}\nonumber\\
\lesssim& \min\{R^{\nu_0+1},R^{\frac{1}{2}}\}
\Big\|\Big(\sum_{\nu\in\Lambda_\infty}\sum_{\ell=1}^{d(\nu)}|b_{\nu,\ell}(\rho) \varphi(\rho)|^2\Big)^{1/2}\Big\|_{L^2_\rho(I)},
\end{align}
and
\begin{align}\label{L-infty1}
&\Big\|\Big(\sum_{\nu\in\Lambda_\infty}\sum_{\ell=1}^{d(\nu)}\Big|\mathcal{H}_{\nu}\big[\rho^{-1}e^{\pm it\rho}\varphi(\rho) b_{\nu,\ell}(\rho)\big](r)\Big|^2\Big)^{1/2}\Big\|_{L^\infty_tL^\infty_{r^{n-1}\mathrm{d}r}(\R\times S_R)}  \nonumber\\
\lesssim& \min\{R^{\nu_0-\frac{n-2}2},R^{-\frac{n-2}{2}-\frac13}\}
\Big\|\Big(\sum_{\nu\in\Lambda_\infty}\sum_{\ell=1}^{d(\nu)}|b_{\nu,\ell}(\rho)|^2\Big)^{1/2}\varphi(\rho)\Big\|_{L^1_\rho(I)},
\end{align}
and
\begin{align}\label{L-infty}
&\Big\|\Big(\sum_{\nu\in\Lambda_\infty}\sum_{\ell=1}^{d(\nu)}\Big|\mathcal{H}_{\nu}\big[\rho^{-1}e^{\pm it\rho}\varphi(\rho) b_{\nu,\ell}(\rho)\big](r)\Big|^2\Big)^{1/2}\Big\|_{L^\infty_tL^\infty_{r^{n-1}\mathrm{d}r}(\R\times S_R)}  \nonumber\\
\lesssim& \min\{R^{\nu_0-\frac{n-2}2},R^{-\frac{n-1}{2}}\}
\Big\|\Big(\sum_{\nu\in\Lambda_\infty}\sum_{\ell=1}^{d(\nu)}|b_{\nu,\ell}(\rho)|^2\Big)^{1/2}\varphi(\rho)\Big\|_{L^2_\rho(I)}.
\end{align}
\end{proposition}

\begin{proof}
 To prove this proposition, we divide into two cases $R\lesssim 1$ and $R\gg1$.
For $R\lesssim 1$, it suffices to prove, for $q\geq2$
\begin{align*}
&\Big\|\Big(\sum_{\nu\in\Lambda_\infty}\sum_{\ell=1}^{d(\nu)}\Big|\int_0^\infty r^{-\frac{n-2}{2}}e^{\pm it\rho}
J_{\nu}(r\rho)b_{\nu,\ell}(\rho)\varphi(\rho)\rho^{\frac{n-2}{2}}\mathrm{d}\rho\Big|^2\Big)^{1/2}\Big\|
_{L^q_tL^q_{r^{n-1}\mathrm{d}r}(\R\times S_R)} \nonumber \\
\lesssim& R^{\frac{n}{q}+\nu_0-\frac{n-2}2}
\Big\|\Big(\sum_{\nu\in\Lambda_\infty}\sum_{\ell=1}^{d(\nu)}|b_{\nu,\ell}(\rho) \varphi(\rho)|^2\Big)^{1/2}\Big\|_{L^2_\rho(I)}.
\end{align*}
To this end, since $q\geq2,$ we use the Minkowski inequality and the Hausdorff-Young inequality in $t$ variable to obtain
\begin{align*}
&\Big\|\Big(\sum_{\nu\in\Lambda_\infty}\sum_{\ell=1}^{d(\nu)}\Big|\int_0^\infty r^{-\frac{n-2}{2}}e^{\pm it\rho}
J_{\nu}(r\rho)b_{\nu,\ell}(\rho)\varphi(\rho)\rho^{\frac{n-2}{2}}\mathrm{d}\rho\Big|^2\Big)^{1/2}\Big\|
_{L^q_tL^q_{r^{n-1}\mathrm{d}r}(\R\times S_R)} \nonumber \\
\lesssim&
\Big\|\Big(\sum_{\nu\in\Lambda_\infty}\sum_{\ell=1}^{d(\nu)}\Big\|r^{-\frac{n-2}{2}}J_{\nu}(r\rho)
b_{\nu,\ell}(\rho)\varphi(\rho)\rho^{\frac{n-2}{2}}\Big\|^2_{L^{q'}_\rho(I)}\Big)^{1/2}\Big\|_{L^q_{r^{n-1}dr}(S_R)}.
\end{align*}
Recall \eqref{est:r} the rough estimate for Bessel function
\begin{equation*}
|J_\nu(r)|\leq\frac{Cr^\nu}{2^\nu\Gamma(\nu+\frac{1}{2})\Gamma(\frac12)}\Big(1+\frac{1}{\nu+\frac12}\Big),
\end{equation*}
then, by using Stirling's formula $\Gamma(\nu+1)\sim\sqrt{\nu}(\nu/e)^\nu$, we obtain
\begin{align*}
&\Big\|\Big(\sum_{\nu\in\Lambda_\infty}\sum_{\ell=1}^{d(\nu)}\Big\|r^{-\frac{n-2}{2}}J_{\nu}(r\rho)
b_{\nu,\ell}(\rho)\varphi(\rho)\rho^{\frac{n-2}{2}}\Big\|^2_{L^{q'}_\rho(I)}\Big)^{1/2}\Big\|_{L^q_{r^{n-1}dr}(S_R)}  \nonumber \\
\lesssim&
R^{\frac{n}{q}+\nu_0-\frac{n-2}2}\Big\|\Big(\sum_{\nu\in\Lambda_\infty}\sum_{\ell=1}^{d(\nu)}|b_{\nu,\ell}(\rho)|^2\Big)^{1/2}
\varphi(\rho)\Big\|_{L^{q'}_\rho(I)},
\end{align*}
where we have used Minkowski's inequality again and $\rho\in I=[1,2]$.
Therefore, by choosing $q=2$ and $q=\infty$ respectively, we have proved \eqref{L-2}, \eqref{L-infty1} and \eqref{L-infty} when $R\lesssim 1$.\vspace{0.2cm}

Next we consider the case $R\gg1$. We first prove \eqref{L-2}. By using the same argument as above (the Minkowski inequality and the Hausdorff-Young inequality in $t$),
we have
\begin{align*}
&\Big\|\Big(\sum_{\nu\in\Lambda_\infty}\sum_{\ell=1}^{d(\nu)}\Big|\int_0^\infty r^{-\frac{n-2}{2}}e^{\pm it\rho}
J_{\nu}(r\rho)b_{\nu,\ell}(\rho)\varphi(\rho)\rho^{\frac{n-2}{2}}\mathrm{d}\rho\Big|^2\Big)^{1/2}\Big\|
_{L^2_tL^2_{r^{n-1}\mathrm{d}r}(\R\times S_R)} \nonumber \\
\lesssim&
R^{1/2}\Big\|\Big(\sum_{\nu\in\Lambda_\infty}\sum_{\ell=1}^{d(\nu)}\Big\|J_{\nu}(r\rho)
b_{\nu,\ell}(\rho)\varphi(\rho)\rho^{\frac{n-2}{2}}\Big\|^2_{L^{2}_\rho(I)}\Big)^{1/2}\Big\|_{L^2_{dr}(S_R)}.
\end{align*}
By using Lemma \ref{lem:bessel}, we can prove
\begin{equation}
\int_R^{2R}|J_\nu(r)|^2 \mathrm{d}r\leq C, \quad R\gg1,
\end{equation}
where the constant $C$ is independent of $R$ and $\nu$. We refer to \cite[(3.21)]{Zhang} for details.
Thus we prove \eqref{L-2} for $R\gg1$.

Next we prove \eqref{L-infty1} which is a consequence of
\begin{align*}
&\Big\|\Big(\sum_{\nu\in\Lambda_\infty}\sum_{\ell=1}^{d(\nu)}\Big|\int_0^\infty J_{\nu}(r\rho)e^{-it\rho}b_{\nu,\ell}(\rho)\varphi(\rho)\rho^{\frac{n-2}{2}}\mathrm{d}\rho
\Big|^2\Big)^{1/2}\Big\|_{L^\infty_tL^\infty_r(\R\times S_R)} \nonumber \\
\lesssim&
R^{-1/3}\Big\|\Big(\sum_{\nu\in\Lambda_\infty}\sum_{\ell=1}^{d(\nu)}|b_{\nu,\ell}(\rho)|^2\Big)^{1/2}\varphi(\rho)\Big\|_{L^1_\rho(I)}.
\end{align*}
This is easily proved by using the Minkowski inequality and the Hausdorff-Young inequality in $t$ variable as before due to the uniform estimate
\begin{equation}
|J_{\nu}(r)|\leq C r^{-1/3}\qquad r\gg 1,
\end{equation}
which is implied by Lemma \ref{lem:bessel}.

Now we prove \eqref{L-infty} which will be implied by
\begin{align*}
&\Big\|\Big(\sum_{\nu\in\Lambda_\infty}\sum_{\ell=1}^{d(\nu)}\Big|\int_0^\infty J_{\nu}(r\rho)e^{-it\rho}b_{\nu,\ell}(\rho)\varphi(\rho)\rho^{\frac{n-2}{2}}\mathrm{d}\rho
\Big|^2\Big)^{1/2}\Big\|_{L^\infty_tL^\infty_r(\R\times S_R)} \nonumber \\
\lesssim&
R^{-1/2}\Big\|\Big(\sum_{\nu\in\Lambda_\infty}\sum_{\ell=1}^{d(\nu)}|b_{\nu,\ell}(\rho)|^2\Big)^{1/2}\varphi(\rho)\Big\|_{L^2_\rho(I)}.
\end{align*}

Recall Schl\"afli's integral representation \eqref{SIR}, we write $J_{\nu}(r\rho)$ as
\begin{align}\label{SIR'}
J_{\nu}(r\rho)&=\frac{1}{2\pi}\int_{-\pi}^\pi e^{ir\rho\sin\theta-i\nu\theta}
\mathrm{d}\theta-\frac{\sin(\nu\pi)}{\pi}\int_0^\infty e^{-(r\sinh s+\nu s)}
\mathrm{d}s \nonumber\\
&=\widetilde{J}_{\nu}(r\rho)-E_{\nu}(r\rho),
\end{align}
where
\begin{equation}\label{est:E'}
 |E_{\nu}(r\rho)|\leq C(r\rho)^{-1}
\end{equation}
with $C$ being independent of $k$ and $R.$
There, by the same argument as before and ~\eqref{est:E'}~, we easily get
\begin{align}\label{est:E-term}
&\Big\|\Big(\sum_{\nu\in\Lambda_\infty}\sum_{\ell=1}^{d(\nu)}\Big|\int_0^\infty e^{-it\rho}E_{\nu}(r\rho)b_{\nu,\ell}(\rho)
\varphi(\rho)\rho^{\frac{n-2}2}\mathrm{d}\rho\Big|^2\Big)^{1/2}\Big\|_{L^\infty_t(\R;L^\infty_r(R/2,R))}\nonumber\\
\lesssim& R^{-1}\Big\|\Big(\sum_{\nu\in\Lambda_\infty}\sum_{\ell=1}^{d(\nu)}|b_{\nu,\ell}(\rho)\varphi(\rho)|^2\Big)^{1/2}\Big\|_{L^1_\rho(I)}.
\end{align}

It thus only remains to prove
\begin{align}\label{re:J}
&\Big\|\Big(\sum_{\nu\in\Lambda_\infty}\sum_{\ell=1}^{d(\nu)}\Big|\int_0^\infty e^{-it\rho}\widetilde{J}_{\nu}(r\rho)
b_{\nu,\ell}(\rho)\varphi(\rho)\rho^{\frac{n-2}2}\mathrm{d}\rho\Big|^2\Big)^{1/2}
\Big\|_{L^\infty_t(\R;L^\infty_r([R/2,R]))}\nonumber\\
\lesssim& R^{-\frac{1}{2}}\Big\|\Big(\sum_{\nu\in\Lambda_\infty}\sum_{\ell=1}^{d(\nu)}|b_{\nu,\ell}(\rho)\varphi(\rho)|^2\Big)^{1/2}\Big\|_{L^2_\rho(I)}.
\end{align}
To use the stationary phase argument, we decompose $[-\pi,\pi]$ into three intervals
\begin{align}\label{decom}
[-\pi,\pi]=I_1\cup I_3\cup I_2\nonumber
\end{align}
where
\begin{equation}\label{decom'}
\begin{split}
I_1&=[-\delta,\delta],\\ I_2&=[-\pi, -\frac{\pi}{2}-\delta]
\cup[\frac{\pi}{2}+\delta,\pi], \\ I_3&=[-\frac{\pi}{2}-\delta,-\delta]
\cup[\delta,\frac{\pi}{2}+\delta]
\end{split}
\end{equation}
with $0<\delta\ll1$ being fixed later.
Let
\begin{equation*}
\Phi_{r,\nu}(\theta)=\sin\theta-\frac{\nu}{r}\theta
\end{equation*}
and a simple computation yields
\begin{equation*}
\Phi'_{r,\nu}(\theta)=\cos\theta-\frac{\nu}{r}, \quad \Phi''_{r,\nu}(\theta)=-\sin\theta.
\end{equation*}
Construct a smooth function $\Lambda_\delta(\theta)$ which is defined by
\begin{align}\label{4.12}
\Lambda_\delta(\theta)=
\begin{cases}
&1, \quad \theta\in I_1,\\
&0, \quad \theta\notin2I_1.
\end{cases}
\end{align}
Therefore,  based on ~\eqref{decom'}, we write $\widetilde{J}_\nu(r)$ as
\begin{align*}
\widetilde{J}_\nu(r)&=\widetilde{J}_\nu^1(r)+\widetilde{J}_\nu^2(r)+\widetilde{J}_\nu^3(r).
\end{align*}
where
\begin{equation*}
\begin{split}
\widetilde{J}_\nu^1(r)
&=\frac{1}{2\pi}\int_{-\pi}^\pi e^{ir\Phi_{r,\nu}(\theta)}\Lambda_\delta(\theta) \mathrm{d}\theta,\\
\widetilde{J}_\nu^2(r)&=\frac{1}{2\pi}\int_{I_2} e^{ir\Phi_{r,\nu}(\theta)}
\mathrm{d}\theta,\\
\widetilde{J}_\nu^3(r)&=
\frac{1}{2\pi}\int_{I_3}e^{ir\Phi_{r,\nu}(\theta)}(1-\Lambda_\delta(\theta))\mathrm{d}\theta.
\end{split}
\end{equation*}
For $\theta\in I_2,$ thus $|\Phi'_{r,\nu}(\theta)|=|\cos\theta-\frac{\nu}{r}|\geq\sin\delta$; When $\theta\in I_3$,  one has $|\Phi''_{r,\nu}(\theta)|\geq\sin\delta$.
By using Van der Corput lemma \ref{lem:VdC},  we have
\begin{equation}\label{4.14}
|\widetilde{J}^2_\nu(r)|
\leq C_\delta r^{-1},\quad |\widetilde{J}^3_\nu(r)|
\leq C_\delta r^{-1/2}.
\end{equation}
Hence, applying the Hausdorff-Young inequality and Minkowski's inequality, we get
\begin{align}\label{4.15}
&\Big\|\Big(\sum_{\nu\in\Lambda_\infty}\sum_{\ell=1}^{d(\nu)}\Big|\int_0^\infty e^{-it\rho}\big(\widetilde{J}_{\nu}^2(r\rho)+|\widetilde{J}^3_\nu(r\rho)\big)b_{\nu,\ell}(\rho)
\varphi(\rho)\rho^{\frac{n-2}2}\mathrm{d}\rho\Big|^2\Big)^{1/2}\Big\|_{L^\infty_t(\R;L^\infty_r(R/2,R))}\nonumber\\
\lesssim& R^{-1/2}\Big\|\Big(\sum_{\nu\in\Lambda_\infty}\sum_{\ell=1}^{d(\nu)}|b_{\nu,\ell}(\rho)|^2\Big)^{1/2}\varphi(\rho)\Big\|_{L^1_\rho(I)}.
\end{align}
So proving \eqref{re:J} is reduced to prove
\begin{align}\label{4.18}
&\Big\|\Big(\sum_{\nu\in\Lambda_\infty}\sum_{\ell=1}^{d(\nu)}\Big|\int_0^\infty e^{-it\rho}\widetilde{J}_{\nu}^1(r\rho)b_{\nu,\ell}(\rho)
\varphi(\rho)\rho^{\frac{n-2}2}\mathrm{d}\rho\Big|^2\Big)^{1/2}\Big\|_{L^\infty_t(\R;L^\infty_r(R/2,R))}\nonumber\\
\lesssim& R^{-1/2}\Big\|\Big(\sum_{\nu\in\Lambda_\infty}\sum_{\ell=1}^{d(\nu)}|b_{\nu,\ell}(\rho)|^2\Big)^{1/2}\varphi(\rho)\Big\|_{L^2_\rho(I)}.
\end{align}
For our purpose, we write the Fourier series of $b_{\nu,\ell}(\rho)$ as
\begin{equation}\label{4.19}
b_{\nu,\ell}(\rho)=\sum_jb_{\nu,l}^je^{i\frac{\pi}{2}\rho j}, \quad
b_{\nu,\ell}^j=\frac{1}{4}\int_0^4b_{\nu,l}(\rho)e^{-i\frac{\pi}{2}\rho j}  \mathrm{d}\rho.
\end{equation}

Therefore we have
\begin{equation}\label{4.24}
\|b_{\nu,\ell}(\rho)\|^2_{L^2_\rho(I)}=\sum_j|b_{\nu,\ell}^j|^2
\end{equation}
and
\begin{align}\label{4.25}
&\int_0^\infty e^{-it\rho}\widetilde{J}_{\nu}^1(r\rho)
b_{\nu,\ell}(\rho)\varphi(\rho)\rho^{\frac{n-2}2}\mathrm{d}\rho\nonumber\\
=&\frac{1}{2\pi}\int_0^\infty e^{-it\rho}\int_{-\pi}^{\pi}e^{ir\rho\sin\theta-i\nu\theta}
\Lambda_\delta(\theta)\sum_jb_{\nu,\ell}^je^{i\frac{\pi}{2}\rho j}\varphi(\rho)\rho^{\frac{n-2}2}\mathrm{d}\rho\mathrm{d}\theta\nonumber\\
\lesssim&\sum_jb_{\nu,\ell}^j\int_{\R^2}e^{2\pi i\rho(r\sin\theta-(t-\frac{j}{4}))}\varphi(\rho)\rho^{\frac{n-2}2}\mathrm{d}\rho e^{-i\nu\theta}\Lambda_\delta(\theta)\mathrm{d}\theta.
\end{align}
Then we estimate the term in ~\eqref{4.25}~.
Let $m=t-\frac{j}{4}$, we write
\begin{align}\label{4.26}
\psi_m^\nu(r)
=&\int_{\R^2}e^{2\pi i\rho(r\sin\theta-m)}\varphi(\rho)\rho^{\frac{n-2}2}\mathrm{d}\rho e^{-i\nu\theta}\Lambda_\delta(\theta)\mathrm{d}\theta\nonumber\\
=&\int_{\R}\check{\varphi}(r\sin\theta-m) e^{-i\nu\theta}\Lambda_\delta(\theta)\mathrm{d}\theta.
\end{align}
It is apparent $\check{\varphi}$ is a Schwartz function, so we have for any $N>0$
\begin{equation}\label{4.27}
|\check{\varphi}(r\sin\theta-m)|\leq C_N(1+|r\sin\theta-m|)^{-N}.
\end{equation}
We consider two cases to study the property of function $\psi_m^\nu(r)$.

{\bf Case 1: $|m|\geq4R.$} Since $r\leq2R\leq|m|$ and $|\theta|\leq2\delta,$
we have
\begin{equation}\label{4.28}
|r\sin\theta-m|\geq|m|-r|\sin\theta|\geq\frac{1}{100}|m|
\end{equation}
and thus
\begin{equation}\label{4.29}
|\psi_m^\nu(r)|\leq C_{\delta,N}(1+|m|)^{-N}.
\end{equation}
Using this inequality, we control \eqref{4.25} by
\begin{equation}\label{4.30}
C_{\delta,N}R^{-N}\Big\|\Big(\sum_{\nu\in\Lambda_\infty}\sum_{\ell=1}^{d(\nu)}\Big|\sum_{j:4R\leq|t-\frac{j}{4}|}
b_{\nu,\ell}^j\Big(1+\Big|t-\frac{j}{4}\Big|\Big)^{-N}
\Big|^2\Big)^{1/2}\Big\|_{L^\infty_t(\R;L^\infty_r(R/2,R))}.
\end{equation}
Applying Cauchy-Schwartz's inequality to the above inequality and then choosing $N$ large enough, we bound ~\eqref{4.30}~ by
\begin{equation}\label{4.31}
C_{\delta,N}R^{-N}\Big(\sum_{\nu\in\Lambda_\infty}\sum_{\ell=1}^{d(\nu)}\sum_j|b_{\nu,\ell}^j|^2\Big)^{1/2}
\lesssim R^{-N}\Big\|\Big(\sum_{\nu\in\Lambda_\infty}\sum_{\ell=1}^{d(\nu)}\Big|b_{\nu,\ell}(\rho)\Big|^2\Big)^{1/2}\varphi(\rho)
\Big\|_{L^2_\rho(I)},
\end{equation}
where we have used ~\eqref{4.24}~ and $\text{supp}\varphi\subset I=[1,2].$

{\bf Case 2:  $|m|<4R.$} We get based on ~\eqref{4.26}~ and ~\eqref{4.27}~
\begin{align*}
|\psi_m^\nu(r)|\leq\frac{C_N}{2\pi}&\Big(\int_{\{\theta: |\theta|<2\delta,|r\sin\theta-m|\leq1\}}\mathrm{d}\theta
\\\nonumber&\quad+\int_{\{\theta: |\theta|<2\delta,|r\sin\theta-m|\geq1\}}
(1+|r\sin\theta-m|)^{-N}\mathrm{d}\theta\Big).
\end{align*}
Making the variable change $y=r\sin\theta-m$, we further have
\begin{align}\label{4.33}
|\psi_m^\nu(r)|\leq\frac{C_N}{2\pi r}\Big(\int_{\{y:|y|\leq1\}}\mathrm{d}y
+\int_{\{y:|y|\geq1\}}(1+|y|)^{-N}\mathrm{d}y\Big)\lesssim r^{-1}.
\end{align}

We put the set $A=\{j\in\mathbb{Z}:|t-\frac{j}{4}|<4R\}$ for fixed $t$ and $R$. Obviously, the cardinality of $A$ is $O(R)$. Then,  from \eqref{4.33} and \eqref{4.24}, we obtain
\begin{align*}
&\Big\|\Big(\sum_{\nu\in\Lambda_\infty}\sum_{\ell=1}^{d(\nu)}\Big|\sum_{j\in A}b^j_{\nu,\ell}\psi^\nu_m(r)\Big|^2\Big)^{1/2}\Big\|_
{L^\infty_tL^\infty_r(\R\times S_R)}  \\
\leq& C_{\delta,N}R^{-\frac{1}{2}}\Big(\sum_{\nu\in\Lambda_\infty}\sum_{\ell=1}^{d(\nu)}\sum_j|b^j_{\nu,\ell}|^2\Big)^{1/2}\\
=&C_{\delta,N}R^{-\frac{1}{2}}\Big(\sum_{\nu\in\Lambda_\infty}\sum_{\ell=1}^{d(\nu)}\|b_{\nu,\ell}(\rho)\|^2_{L^2_\rho}\Big)^{1/2}\\
\lesssim& R^{-\frac{1}{2}}\Big\|\Big(\sum_{\nu\in\Lambda_\infty}\sum_{\ell=1}^{d(\nu)}|b_{\nu,\ell}(\rho)|^2\Big)^{1/2}\Big\|_{L^2_\rho(I)}.
\end{align*}

\end{proof}

As a consequence of the interpolation and Proposition \ref{LRE}, we obtain
\begin{proposition}\label{LRE'}
For $q\geq2$, we have
\begin{align}\nonumber
&\Big\|\Big(\sum_{\nu\in\Lambda_\infty}\sum_{\ell=1}^{d(\nu)}\Big|\int_0^\infty r^{-\frac{n-2}{2}}J_{\nu}(r\rho)e^{-it\rho}b_{\nu,\ell}(\rho)\varphi(\rho)\rho^{\frac{n-2}{2}}\mathrm{d}\rho\Big|^2\Big)^{1/2}\Big\|_
{L^q_tL^q_{r^{n-1}\mathrm{d}r}(\R\times S_R)}  \\\label{LRE'1}
\lesssim&
\min\{R^{\frac{n}{q}+\nu_0-\frac{n-2}2},R^{-\frac{3n-4}{6}[1-\frac{2(3n-1)}{(3n-4)q}]}\}\Big\|\Big(\sum_{\nu\in\Lambda_\infty}
\sum_{\ell=1}^{d(\nu)}|
b_{\nu,\ell}(\rho)|^2\Big)^{1/2}\varphi(\rho)\Big\|_{L^{q'}_\rho(I)},
\end{align}
and
\begin{align}\nonumber
&\Big\|\Big(\sum_{\nu\in\Lambda_\infty}\sum_{\ell=1}^{d(\nu)}\Big|\int_0^\infty r^{-\frac{n-2}{2}}J_{\nu}(r\rho)e^{-it\rho}b_{\nu,\ell}(\rho)\varphi(\rho)\rho^{\frac{n-2}{2}}\mathrm{d}\rho\Big|^2\Big)^{1/2}\Big\|_
{L^q_tL^q_{r^{n-1}\mathrm{d}r}(\R\times S_R)}  \\\label{LRE'2}
\lesssim&
\min\{R^{\frac{n}{q}+\nu_0-\frac{n-2}2},R^{-\frac{n-1}{2}[1-\frac{2n}{(n-1)q}]}\}\Big\|\Big(\sum_{\nu\in\Lambda_\infty}\sum_{\ell=1}^{d(\nu)}|
b_{\nu,\ell}(\rho)|^2\Big)^{1/2}\varphi(\rho)\Big\|_{L^2_\rho(I)}.
\end{align}

\end{proposition}




\section{Proof of Theorem \ref{thm:main}}

In this section, we prove Theorem \ref{thm:main} based on the estimates of Hankel transform in Proposition \ref{LRE'}.
In the end of this section, we construct an counterexample to show the necessity of \eqref{add-cond}. \vspace{0.2cm}

From \eqref{sol}, it suffices to estimate
\begin{align}
&\|u(t,r,\theta)\|_{L^{q}_{t}(\R; L^q_{r^{n-1}dr}((0,\infty);L^{2}_\theta(Y)))}\\\nonumber
\lesssim& \sum_{\pm} \Big\|\Big(\sum_{\nu\in\Lambda_\infty}\sum_{\ell=1}^{d(\nu)}
\Big|\int_0^\infty r^{-\frac{n-2}{2}}J_{\nu}(r\rho)e^{\pm it\rho}
b_{\nu,\ell}(\rho)\rho^{\frac{n-2}{2}}\mathrm{d}\rho\Big|^2\Big)^\frac12\Big\|_{L^q_tL^q_{r^{n-1}\;dr}(\R\times\R_+)}
\end{align}
By the symmetry of $t$, we only need to consider one of signs $\pm$. We only consider the minus sign and apply dyadic decompositions
to obtain
\begin{align*}
&\Big\|\Big(\sum_{\nu\in\Lambda_\infty}\sum_{\ell=1}^{d(\nu)}
\Big|\int_0^\infty r^{-\frac{n-2}{2}}J_{\nu}(r\rho)e^{-it\rho}
b_{\nu,\ell}(\rho)\rho^{\frac{n-2}{2}} \mathrm{d}\rho\Big|^2\Big)^\frac12\Big\|_{L^q_tL^q_{r^{n-1}\;dr}(\R\times\R_+)}
\\\nonumber
\lesssim& \Big(\sum_R\Big(\sum_M\Big\|\Big(\sum_{\nu\in\Lambda_\infty}\sum_{\ell=1}^{d(\nu)}\Big|\int_0^\infty r^{-\frac{n-2}{2}}J_{\nu}(r\rho) e^{-it\rho}
b_{\nu,\ell}(\rho)\varphi(\frac{\rho}{M})\rho^{\frac{n-2}{2}}\;d\rho
\Big|^2\Big)^\frac12\Big\|_{L^q_tL^q_{r^{n-1}\;dr}(\R\times S_R)}\Big)^q\Big)^{1/q}
\end{align*}
where both $R$ and $M$ are dyadic numbers, $\varphi\in C^\infty_c([1,2])$ values in $[0,1]$ such that $\sum\limits_{M\in2^{\Z}}\varphi(\rho/M)=1$. Furthermore, by scaling argument, we obtain
\begin{align}\label{3.12}
&\|u(t,r,\theta)\|_{L^{q}_{t}(\R; L^q_{r^{n-1}dr}((0,\infty);L^{2}_\theta(Y)))}\\\nonumber
\lesssim& \Big(\sum_R\Big(\sum_MM^{(n-1)-\frac{n+1}{q}}\Big\|
\Big(\sum_{\nu\in\Lambda_\infty}\sum_{\ell=1}^{d(\nu)}\Big|\int_0^\infty r^{-\frac{n-2}{2}}J_{\nu}(r\rho)\\\nonumber
&\qquad\qquad \times e^{-it\rho}
b_{\nu,\ell}(M\rho)\varphi(\rho) \rho^{\frac{n-2}{2}}\;d\rho\Big|^2\Big)^\frac12
\Big\|_{L^q_tL^q_{r^{n-1}\;dr}(\R\times S_{MR})}\Big)^q\Big)^{1/q}.
\end{align}

For our purpose, we divide into two cases.

$\bullet$ {\bf Case1:} $1\leq p\leq2$. By \eqref{LRE'1}, we obtain
\begin{align}\label{3.13}
&\|u(t,r,\theta)\|_{L^{q}_{t}(\R; L^q_{r^{n-1}dr}((0,\infty);L^{2}_\theta(Y)))}\\\nonumber
\lesssim& \Big(\sum_R\Big(\sum_MM^{(n-1)-\frac{n+1}{q}}
\min\{(RM)^{\frac{n}{q}+\nu_0-\frac{n-2}2},(RM)^{-\frac{3n-4}{6}[1-\frac{2(3n-1)}{(3n-4)q}]}\}
\\\nonumber
&\qquad\qquad\times \Big\|\Big(\sum_{\nu\in\Lambda_\infty}\sum_{\ell=1}^{d(\nu)}|b_{\nu,\ell}(M\rho)|^2\Big)^\frac12\varphi(\rho)\Big\|_{L^{p}_\rho(I)}
\Big)^q\Big)^{1/q}
\end{align}
where we use the fact that  $$\frac{n+1}q=\frac{n-1}{p'}\Leftrightarrow \frac1{q'}= \frac{1}p+\frac{2}{p'(n+1)}$$ implies $p\geq q'$; On
the other hand, since $1\leq p\leq 2$ and $n\geq2$, one has
\begin{equation}\label{p<2}
q\geq \frac{2(n+1)}{n-1}>\frac{2(3n-1)}{3n-4}.
\end{equation}

If $\nu_0\geq \tfrac{n-2}2$, the fact \eqref{p<2} is enough to guarantee
\begin{equation}\label{S-cond1}
\sup_{R>0}\sum_{M}\min\big\{(RM)^{\frac{n}{q}+\nu_0-\frac{n-2}2},
(RM)^{-\frac{3n-4}{6}[1-\frac{2(3n-1)}{(3n-4)q}]}\big\}<\infty,
\end{equation}
and
\begin{equation}\label{S-cond2}
\sup_{M>0}\sum_{R}\min\big\{(RM)^{\frac{n}{q}+\nu_0-\frac{n-2}2},
(RM)^{-\frac{3n-4}{6}[1-\frac{2(3n-1)}{(3n-4)q}]}\big\}<\infty.
\end{equation}
However, if $0<\nu_0\leq \tfrac{n-2}2$, we need \eqref{add-cond} to ensure \eqref{S-cond1} and \eqref{S-cond2} to be true.
Therefore, by Schur test's lemma,  we show
\begin{align}\label{3.14}
&\|u(t,r,\theta)\|_{L^{q}_{t}(\R; L^q_{r^{n-1}dr}((0,\infty);L^{2}_\theta(Y)))}\nonumber\\
\lesssim& \Big(\sum_M\Big\|\Big(\sum_{\nu\in\Lambda_\infty}\sum_{\ell=1}^{d(\nu)}
\Big|b_{\nu,\ell}(\rho)\Big|^2\Big)^\frac12\varphi(\frac{\rho}{M})\rho^{\frac{n-2}{p}}
\Big\|^p_{L^p_\rho}\Big)^{1/p}\nonumber\\
\lesssim&\Big\|\Big(\sum_{\nu\in\Lambda_\infty}\sum_{\ell=1}^{d(\nu)}
|b_{\nu,\ell}(\rho)|^2\Big)^\frac12\rho^{-\frac1p}\Big\|_{L^p_{\rho^{n-1}\;d\rho}(\R_+)}.
\end{align}

$\bullet$ {\bf Case 2:} $p\geq2$. By \eqref{LRE'2}, we have
\begin{align}\label{3.13}
&\|u(t,r,\theta)\|_{L^{q}_{t}(\R; L^q_{r^{n-1}dr}((0,\infty);L^{2}_\theta(Y)))}\\\nonumber
\lesssim& \Big(\sum_R\Big(\sum_MM^{(n-1)-\frac{n+1}{q}}
\min\{(RM)^{\frac{n}{q}+\nu_0-\frac{n-2}2},(RM)^{-\frac{n-1}{2}[1-\frac{2n}{(n-1)q}]}\}\\\nonumber
&\qquad\qquad\times \Big\|\Big(\sum_{\nu\in\Lambda_\infty}\sum_{\ell=1}^{d(\nu)}|b_{\nu,\ell}(M\rho)|^2\Big)^\frac12\varphi(\rho)\Big\|_{L^2_\rho(I)}\Big)^q\Big)^{1/q}.
\end{align}

By noting that $\frac{n+1}q=\frac{n-1}{p'}$ and $p\geq 2$ and using scaling, we have
\begin{align}
&\|u(t,r,\theta)\|_{L^{q}_{t}(\R; L^q_{r^{n-1}dr}((0,\infty);L^{2}_\theta(Y)))}\\\nonumber
\lesssim& \Big(\sum_R\Big(\sum_MM^{(n-1)-\frac{n+1}{q}}
\min\{(RM)^{\frac{n}{q}+\nu_0-\frac{n-2}2},(RM)^{-\frac{n-1}{2}[1-\frac{2n}{(n-1)q}]}\}\\\nonumber
&\qquad\qquad\times \Big\|\Big(\sum_{\nu\in\Lambda_\infty}\sum_{\ell=1}^{d(\nu)}|b_{\nu,\ell}(M\rho)|^2\Big)^\frac12\varphi(\rho)\rho^{\frac{n-2}p}
\Big\|_{L^p_\rho(I)}\Big)^q\Big)^{1/q}
\\\nonumber
\lesssim& \Big(\sum_R\Big(\sum_M\min\{(RM)^{\frac{n}{q}+\nu_0-\frac{n-2}2},(RM)^{-\frac{n-1}{2}[1-\frac{2n}{(n-1)q}]}\}\\\nonumber
&\qquad\qquad\times \Big\|\Big(\sum_{\nu\in\Lambda_\infty}\sum_{\ell=1}^{d(\nu)}|b_{\nu,\ell}(\rho)|^2\Big)^\frac12\varphi(\frac{\rho}M)\rho^{\frac{n-2}p}
\Big\|_{L^p_\rho(\R)}\Big)^q\Big)^{1/q}.
\end{align}

If $\nu_0\geq \tfrac{n-2}2$, the condition $q>\frac{2n}{n-1}$ in \eqref{condition'} is enough to guarantee
\begin{equation}\label{S-cond1'}
\sup_{R>0}\sum_{M}\min\big\{(RM)^{\frac{n}{q}+\nu_0-\frac{n-2}2},
(RM)^{-\frac{n-1}{2}[1-\frac{2n}{q(n-1)}]}\big\}<\infty,
\end{equation}
and
\begin{equation}\label{S-cond2'}
\sup_{M>0}\sum_{R}\min\big\{(RM)^{\frac{n}{q}+\nu_0-\frac{n-2}2},
(RM)^{-\frac{n-1}{2}[1-\frac{2n}{q(n-1)}]}\big\}<\infty.
\end{equation}
However, if $0<\nu_0\leq \tfrac{n-2}2$, we need \eqref{add-cond} again to ensure \eqref{S-cond1'} and \eqref{S-cond2'} to be true.
Therefore, by Schur's lemma and $\ell^{p} \hookrightarrow\ell^q$ since $q>\frac{2n}{n-1}>p$, we show
\begin{align}\label{3.14}
&\|u(t,r,\theta)\|_{L^{q}_{t}(\R; L^q_{r^{n-1}dr}((0,\infty);L^{2}_\theta(Y)))}\nonumber\\
\lesssim& \Big(\sum_M\Big\|\Big(\sum_{\nu\in\Lambda_\infty}\sum_{\ell=1}^{d(\nu)}
\Big|b_{\nu,\ell}(\rho)\Big|^2\Big)^\frac12\varphi(\frac{\rho}{M})\rho^{\frac{n-2}{p}}
\Big\|^p_{L^p_\rho}\Big)^{1/p}\nonumber\\
\lesssim&\Big\|\Big(\sum_{\nu\in\Lambda_\infty}\sum_{\ell=1}^{d(\nu)}
|b_{\nu,\ell}(\rho)|^2\Big)^\frac12\rho^{-\frac1p}\Big\|_{L^p_{\rho^{n-1}\;d\rho}(\R_+)}.
\end{align}
In sum,  by orthogonality formula \eqref{orth}, we prove
\begin{align}
\|u(t,r,\theta)\|_{L^{q}_{t}(\R; L^q_{r^{n-1}dr}((0,\infty);L^{2}_\theta(Y)))}\lesssim\Big\|\rho^{-\frac1p}\hat{f}(\rho,\omega)\Big\|_{L^p_{\rho^{n-1}\;d\rho}
([0,\infty);L^2_\omega(Y))}.
\end{align}

In the end of this section,  we show the necessity of assumption \eqref{add-cond} by constructing a counterexample.
We will prove

\begin{proposition}
Let $\nu_0$ be in Theorem \ref{thm:main} and let $q$ satisfy \eqref{condition'} but $q\geq \frac{2n}{n-2-2\nu_0}$. Then there exists a counterexample such that the inequality $\eqref{est:restriction}$ fails.

\end{proposition}

\begin{proof} We use the argument of \cite{ZZ2} to construct a counterexample.
Choose $\chi(\rho)\in C_c^\infty([1,2])$ to value in $[0,1]$, we take the initial data $f=(\mathcal{H}_{\nu_0}\chi)(r)$, which is independent of the angular variable $\theta$. Then
the distorted Fourier transform of $f$ is the Hankel transform. Therefore, we obtain
\begin{align}
\big\|\rho^{-\frac1p}\hat{f}(\rho,\omega)\big\|_{L^p_{\rho^{n-1}\;d\rho}([0,\infty);L^2_\omega(Y))}=\big\|\rho^{-\frac1p}\chi(\rho)
\big\|_{L^p_{\rho^{n-1}\;d\rho}([0,\infty);L^2_\omega(Y))}<\infty.
\end{align}
 Since $q\geq \frac{2n}{n-2-2\nu_0}$, one has $\frac{1}{q}\leq\frac{1}{2}-\frac{1+\nu_0}{n}$. To lead a contradiction, we will show
\begin{equation}\label{contradiction}
\|u(t,r,\theta)\|_{L^{q}_{t}(\R; L^q_{r^{n-1}dr}((0,\infty);L^{2}_\theta(Y)))}=\infty,\quad \frac{1}{q}\leq\frac{1}{2}-\frac{1+\nu_0}{n}
\end{equation}
where $u(t,r,\theta)$ solves \eqref{equ:Lw}, that is,
\begin{align}
 u(t,r,\theta)=\int_0^\infty(r\rho)^{-\frac{n-2}2}J_{\nu_0}(r\rho)\sin (t\rho) \chi(\rho)\rho^{n-2}\;d\rho .
\end{align}
To prove \eqref{contradiction},
we recall the behavior of  $J_\nu(r)$ as $r\to 0+$. For the complex number $\Re(\nu)>-1/2$, see \cite[Section B.6]{G}, then we have that
\begin{equation}\label{Bessel}
J_{\nu}(r)=\frac{r^\nu}{2^\nu\Gamma(\nu+1)}+S_\nu(r)
\end{equation}
where
\begin{equation*}
S_{\nu}(r)=\frac{(r/2)^{\nu}}{\Gamma\left(\nu+\frac12\right)\Gamma(1/2)}\int_{-1}^{1}(e^{isr}-1)(1-s^2)^{(2\nu-1)/2}\mathrm{d
}s
\end{equation*}
satisfies
\begin{equation*}
|S_{\nu}(r)|\leq \frac{2^{-\Re\nu}r^{\Re\nu+1}}{(\Re\nu+1)|\Gamma(\nu+\frac12)|\Gamma(\frac12)}.\end{equation*}
Now we compute for any $0<\epsilon\ll 1$
\begin{align*}
&\|u(t,r,\theta)\|_{L^{q}_{t}(\R; L^q_{r^{n-1}dr}((0,\infty);L^{2}_\theta(Y)))}\\
=&\text{Vol}(Y)^{1/2}\left\|\int_0^\infty(r\rho)^{-\frac{n-2}2}J_{\nu_0}(r\rho)\sin(t\rho)\chi(\rho)\rho^{n-2}\mathrm{d}\rho
\right\|_{L^q(\R;L^q_{r^{n-1}dr}(0,\infty))}\\
\geq &c\left\|\int_0^\infty(r\rho)^{-\frac{n-2}2}J_{\nu_0}(r\rho)\sin(t\rho)\chi(\rho)\rho^{n-2}\mathrm{d}\rho
\right\|_{L^q([\pi/6,\pi/4];L^q_{r^{n-1}dr}[\epsilon,1])}\\
\geq& c\left\|\int_0^\infty(r\rho)^{-\frac{n-2}2}(r\rho)^{\nu_0}\sin(t\rho)\chi(\rho)\rho^{n-2}\mathrm{d}\rho
\right\|_{L^q([\pi/6,\pi/4];L^q_{r^{n-1}dr}[\epsilon,1])}\\
&-\left\|\int_0^\infty(r\rho)^{-\frac{n-2}2}S_{\nu_0}(r\rho)\sin(t\rho)\chi(\rho)\rho^{n-2}\mathrm{d}\rho
\right\|_{L^q([\pi/6,\pi/4];L^q_{r^{n-1}dr}[\epsilon,1])}.
\end{align*}
We first observe that
\begin{equation}
\begin{split}
&\left\|\int_0^\infty(r\rho)^{-\frac{n-2}2}S_{\nu_0}(r\rho)\sin(t\rho)\chi(\rho)\rho^{n-2}\mathrm{d}\rho
\right\|_{L^q([\pi/6,\pi/4];L^q_{r^{n-1}dr}[\epsilon,1])}\\
\leq& C\left\|\int_0^\infty(r\rho)^{-\frac{n-2}2}(r\rho)^{\nu_0+1}\chi(\rho)\rho^{n-2}\mathrm{d}\rho
\right\|_{L^q([\pi/6,\pi/4];L^q_{r^{n-1}dr}[\epsilon,1])}\\
\leq& C\max\big\{\epsilon^{\nu_0+1-\frac{n-2}2+\frac nq},1\big\}\\
\end{split}
\end{equation}
Next we estimate the lower boundness
\begin{align*}
&\left\|\int_0^\infty(r\rho)^{-\frac{n-2}2}(r\rho)^{\nu_0}\sin(t\rho)\chi(\rho)\rho^{n-2}\mathrm{d}\rho
\right\|_{L^q([\pi/6,\pi/4];L^q_{r^{n-1}dr}[\epsilon,1])}\\
=&\left(\int_{\pi/6}^{\pi/4}\int_{\epsilon}^1 \left|\int_0^\infty(r\rho)^{-\frac{n-2}2}(r\rho)^{\nu_0}\sin(t\rho)\chi(\rho)\rho^{n-2}\mathrm{d}\rho\right|^q r^{n-1}drdt\right)^{1/q}\\
=&C\left(\int_{\pi/6}^{\pi/4}\left|\int_0^\infty\rho^{-\frac{n-2}2}\rho^{\nu_0}\sin(t\rho)\chi(\rho)\rho^{n-2}
\mathrm{d}\rho\right|^{q}dt\right)^{1/q}\times\begin{cases}\epsilon^{\nu_0-\frac{n-2}2+\frac nq} \quad\text{if}\quad \frac1q<\frac12-\frac{\nu_0+1}{n}\\
\ln\epsilon \quad\text{if}\quad \frac1q=\frac12-\frac{\nu_0+1}{n}
\end{cases}\\
\geq& c\begin{cases}\epsilon^{\nu_0-\frac{n-2}2+\frac nq} \quad\text{if}\quad \frac1q<\frac12-\frac{\nu_0+1}{n}\\
\ln\epsilon \quad\text{if}\quad \frac1q=\frac12-\frac{\nu_0+1}{n}
\end{cases}
\end{align*}
where we have used the fact that $\sin(\rho t)\geq 1/2$ for $t\in [\pi/6, \pi/4]$ and $\rho\in [1,2]$, and
\begin{equation}
\begin{split}
\left|\int_0^\infty\rho^{-\frac{n-2}2}\rho^{\nu_0}\sin(t\rho)\chi(\rho)\rho^{n-2}\mathrm{d}\rho\right|\geq \frac12\int_0^\infty\rho^{-\frac{n-2}2}\rho^{\nu_0}\chi(\rho)\rho^{n-2}\mathrm{d}\rho\geq c.
\end{split}
\end{equation}
Hence, we obtain
\begin{equation}
\begin{split}
\|u(t,r,\theta)\|_{L^{q}_{t}(\R; L^q_{r^{n-1}dr}((0,\infty);L^{2}_\theta(Y)))}&\geq c\epsilon^{\nu_0-\frac{n-2}2+\frac nq}-C\max\big\{\epsilon^{\nu_0+1-\frac{n-2}2+\frac nq},1\big\}
\\&\geq c\epsilon^{\nu_0-\frac{n-2}2+\frac nq}\to +\infty \quad \text{as}\quad \epsilon\to 0
\end{split}
\end{equation}
when $\frac1q<\frac12-\frac{\nu_0+1}n.$ And when $\frac1q=\frac12-\frac{\nu_0+1}n$, we get
\begin{align*}
\|u(t,r,\theta)\|_{L^{q}_{t}(\R; L^q_{r^{n-1}dr}((0,\infty);L^{2}_\theta(Y)))}\geq c\ln\epsilon-C\to +\infty \quad \text{as}\quad \epsilon\to 0.
\end{align*}

\end{proof}

We conclude this section by proving \eqref{est:stri'} in Remark \ref{rem:stri}.
\begin{proof}[The proof of  \eqref{est:stri'}] If $\text{supp}~\hat{f}\subset \{\rho\sim M\}$,
then by \eqref{LRE'2}, we have
\begin{align}\label{est:stri''}
&\|u(t,r,\theta)\|_{L^{q}_{t}(\R; L^q_{r^{n-1}dr}((0,\infty);L^{2}_\theta(Y)))}\\\nonumber
\lesssim& M^{\frac{n-1}2-\frac{n+1}q} \Big(\sum_R\Big(\min\{(RM)^{\frac{n}{q}+\nu_0-\frac{n-2}2},(RM)^{-\frac{n-1}{2}[1-\frac{2n}{(n-1)q}]}\}\\\nonumber
&\qquad\qquad\times \Big\|\Big(\sum_{\nu\in\Lambda_\infty}\sum_{\ell=1}^{d(\nu)}|b_{\nu,\ell}(\rho)|^2\Big)^\frac12\varphi(\frac{\rho}M)\rho^{\frac{n-2}2}
\Big\|_{L^2_\rho(\R)}\Big)^q\Big)^{1/q}.
\end{align}
By using the assumption $q>\frac{2n}{n-1}$ when $\nu_0>(n-2)/2$ and $\frac{2n}{n-2-2\nu_0}>q>\frac{2n}{n-1}$ when $0<\nu_0\leq (n-2)/2$,
we see the summation in $R$ converges. Thus we obtain
\begin{align}
\|u(t,r,\theta)\|_{L^{q}_{t}(\R; L^q_{r^{n-1}dr}((0,\infty);L^{2}_\theta(Y)))}
\lesssim M^{\frac{n-1}2-\frac{n+1}q}\|f\|_{\dot H^{-\frac12}},
\end{align}
which is \eqref{est:stri'}.
\end{proof}

\section{The proof of Theorem \ref{thm:KSS}}

In this section, we prove  Theorem \ref{thm:KSS} by using the above argument when $q=2$ and a slight modify the original argument for the wave equation in \cite{KSS}.
\vspace{0.2cm}

We first prove \eqref{est:locendec}. From \eqref{sol} again, it suffices to estimate
\begin{align*}
&\Big\|\Big(\sum_{\nu\in\Lambda_\infty}\sum_{\ell=1}^{d(\nu)}
\Big|\int_0^\infty r^{-\frac{n-2}{2}}J_{\nu}(r\rho)e^{-it\rho}
b_{\nu,\ell}(\rho)\rho^{\frac{n-2}{2}} \mathrm{d}\rho\Big|^2\Big)^\frac12\Big\|_{L^2_tL^2_{r^{n-1}\;dr}(\R\times (0,{\bf R}])}\\
\lesssim& \Big(\sum_{R\leq {\bf R}} \sum_M\sum_{\nu\in\Lambda_\infty}\sum_{\ell=1}^{d(\nu)}\Big\|r^{-\frac{n-2}2}J_{\nu}(r\rho)
b_{\nu,\ell}(\rho)\varphi(\frac{\rho}{M})\rho^{\frac{n-2}{2}}
\Big\|^2_{L^2_{d\rho} L^2_{r^{n-1}dr}(\R\times S_R)}\Big)^{1/2}
\\\nonumber
\lesssim& \Big(\sum_{R\leq {\bf R}} \sum_M M^{n-1}M^{-2}\sum_{\nu\in\Lambda_\infty}\sum_{\ell=1}^{d(\nu)}\Big\|r^{-\frac{n-2}2}J_{\nu}(r\rho)
b_{\nu,\ell}(M\rho)\varphi(\rho)\rho^{\frac{n-2}{2}}
\Big\|^2_{L^2_{d\rho} L^2_{r^{n-1}dr}(\R\times S_{MR})}\Big)^{1/2}
\end{align*}
By the proof of \eqref{L-2}, we have
\begin{align*}\nonumber
& \Big(\sum_{R\leq {\bf R}}  \sum_M M^{n-1}M^{-2}\sum_{\nu\in\Lambda_\infty}\sum_{\ell=1}^{d(\nu)}\Big\|r^{-\frac{n-2}2}J_{\nu}(r\rho)
b_{\nu,\ell}(M\rho)\varphi(\rho)\rho^{\frac{n-2}{2}}
\Big\|^2_{L^2_{d\rho} L^2_{r^{n-1}dr}(\R\times S_{MR})}\Big)^{1/2}\\\nonumber
\lesssim&  \Big(\sum_{R\leq {\bf R}} \sum_M M^{-2} \min\big\{(MR)^{2(1+\nu_0)},RM\big\}\sum_{\nu\in\Lambda_\infty}\sum_{\ell=1}^{d(\nu)}\Big\|
b_{\nu,\ell}(\rho)\varphi(\rho/M)\rho^{\frac{n-2}{2}}
\Big\|^2_{L^2_{d\rho}}\Big)^{1/2}\\\nonumber
\lesssim& {\bf R} \Big( \sum_M M^{-2}  M\sum_{\nu\in\Lambda_\infty}\sum_{\ell=1}^{d(\nu)}\Big\|
b_{\nu,\ell}(\rho)\varphi(\rho/M)\rho^{\frac{n-2}{2}}
\Big\|^2_{L^2_{d\rho}}\Big)^{1/2}\\
\lesssim& {\bf R} \Big( \sum_M \sum_{\nu\in\Lambda_\infty}\sum_{\ell=1}^{d(\nu)}\Big\|
b_{\nu,\ell}(\rho)\varphi(\rho/M)\rho^{-1}
\Big\|^2_{L^2_{\rho^{n-1}d\rho}}\Big)^{1/2}
\end{align*}
where we have used the fact that
\begin{align*}
\sum_{R\leq {\bf R}}  \min\{(MR)^{2(1+\nu_0)},RM\}\lesssim {\bf R} M.
\end{align*}
Then we prove
\begin{align}\label{equ:localres}
\sup_{{\bf R}>0}{\bf R}^{-\frac12}\|u(t,r,\theta)\|_{L^{2}_{t}(\R; L^2_{r^{n-1}dr}((0,{\bf R}];L^{2}_\theta(Y)))}
\lesssim \|f\|_{\dot H^{-1}}.
\end{align}

Next we prove \eqref{kss}. To this end, we first consider the case $\beta>\tfrac12$: in this range we have, by applying \eqref{equ:localres},
\begin{equation*}
\begin{split}
&\| \langle r\rangle^{-\beta} u(t,r,\theta) \|_{L^2([0,T];L^2(X))}\\
\lesssim& \sum_{j\geq0}2^{-j\beta}\|u(t,r,\theta)\|_{L^2([0,T]; L^2((0,2^{j+1}]\times Y))}\\
\lesssim&
\sum_{j\geq 0}2^{j(1/2-\beta)}\|f\|_{\dot H^{-1}}
\lesssim \|f\|_{\dot H^{-1}}.\end{split}
\end{equation*}

Now we consider the case $0\leq\beta\leq \tfrac12$; we divide into two cases.

{\bf Case 1: $T\leq 1$.}  Here, the estimate \eqref{kss} is weaker than the energy estimate
\begin{equation}\label{est:energy}
\|u(t,r,\theta) \|_{L^\infty_tL^2(X)}\leq \|f\|_{\dot{H}^{-1}},
\end{equation}
so that we can immediately write
$$
\|\langle r\rangle^{-\beta} u(t,r,\theta)\|_{L^2([0,T];L^2(X))}\lesssim T^{1/2}\|u(t,r,\theta)\|_{L^2([0,T];L^2(X))}\leq C_\beta(T)\|f\|_{\dot{H}^{-1}}.
$$

{\bf Case 2: $T\geq1$.}  we can use energy estimate \eqref{est:energy} to control on the region $\{r:r\geq T\}$ as follows
\begin{equation*}
\begin{split}
\|\langle r\rangle^{-\beta} u(t,r,\theta)\|_{L^2([0,T]; L^2((T,\infty)\times Y))} &\lesssim T^{-\beta}\|u(t,r,\theta) \|_{L^2([0,T];L^2(X))}\\&
\leq T^{\frac12-\beta}\|f\|_{\dot{H}^{-1}}\leq C_\beta(T)\|f\|_{\dot{H}^{-1}}.
\end{split}
\end{equation*}
While in the region that $\{r:r\leq T\}$, we estimate
\begin{align*}
 &\| \langle r\rangle^{-\beta} u(t,r,\theta) \|_{L^2([0,T];L^2((0,T]\times Y)}^2\\
\lesssim&\sum_{j=0}^{\ln(T+2)}2^{-2j\beta}\|u(t,r,\theta)\|_{L^2([0,T]; L^2((0,2^{j+1}]\times Y))}^2\\
\lesssim&
\sum_{j=0}^{\ln(T+2)}2^{j(1-2\beta)}\|f\|_{\dot H^{-1}}^2\\
\lesssim& \|f\|_{\dot H^{-1}}^2\times \begin{cases}
  T^{1-2\beta}\quad\text{if}\quad 0\leq\beta<\frac12\\
  \log(2+T)\quad\text{if}\quad \beta=\frac12
\end{cases}\\
\lesssim& C_\beta(T)^2\|f\|_{\dot H^{-1}}^2
\end{align*}
which is accepted.

Finally, we turn to
prove \eqref{kss2}. In the region $\{r:\;r\geq T\}$, we utilize the energy estimate \eqref{est:energy} to obtain
\begin{equation}\label{equ:outballset}
  \big\||r|^{-\beta} u(t,r,\theta)\big\|_{L^2([0,T]; L^2([T,\infty)\times Y))}\\ \lesssim T^{-\beta}\|u(t,r,\theta) \|_{L^2_TL^2_x}
\lesssim T^{\frac12-\beta}\|f\|_{\dot{H}^{-1}}.
\end{equation}
In the  region $\{r:\;r\leq T\}$, by \eqref{equ:localres} and $0\leq \beta<\tfrac12$, we get
\begin{align*}
 &\big\||r|^{-\beta} u(t,r,\theta)\big\|_{L^2([0,T];L^2((0,T]\times Y))}\\
\lesssim&\sum_{j=-\infty}^{\log_2 T}2^{-j\beta}\|u(t,r,\theta)\|_{L^2([0,T]; L^2([2^{j},2^{j+1}]\times Y))}\\
\lesssim&
\sum_{j=-\infty}^{\log_2 T}2^{j(\frac12-\beta)}\|f\|_{\dot H^{-1}}\\
\lesssim& T^{\frac12-\beta}\|f\|_{\dot H^{-1}}
\end{align*}
which implies \eqref{kss2}.

\begin{center}

\end{center}


\begin{thebibliography}{99}

\bibitem{Barcelo} B. Barcelo. On the restriction of the Fourier transform to a conical surface, Trans. Amer.
Math. Soc., 292(1985),321-333.

%
%
%


\bibitem{BPSS} N. Burq, F. Planchon, J. Stalker and A. S.
Tahvildar-Zadeh. Strichartz estimates for the wave and Schr\"odinger
equations with the inverse-square potential. J. Funct. Anal., 203
(2003), 519-549.

%

\bibitem{BFM} M. D. Blair, G. A. Ford, and J. L. Marzuola, Strichartz estimates for the wave equation on flat cones, IMRN, 2012, 30 pages, doi:10.1093/imrn/rns002.



\bibitem{CRS}  A. Carbery, E. Romera and F. Soria, Radial weights and mixed norm inequalities for the disc multiplier, J. Funct. Anal.,
 109(1992), 52-75.


\bibitem{CG}  L. D. Carli and L. Grafakos, On the restriction conjecture, Michigan Math. J., 52(2004), 163-180.


\bibitem{CT1} J. Cheeger and M. Taylor,  On the diffraction of waves by conical singularities, I. Comm. Pure Appl.
Math., 35(1982), 275-331.


\bibitem{CT2} J. Cheeger and M. Taylor, On the diffraction of waves by conical singularities, II. Commun. Pure Appl.
Math., 35(1982), 487-529.





%


\bibitem{Cor} A. C\'odoba,  The disc
multipliers,  Duke Math. J., 58(1989),21-29.

\bibitem{CL1}  A. C\'ordoba and  E. Latorre,  Radial multipliers and restriction to surfaces of the Fourier transform in mixed-norm spaces,  Math. Z.,  286(2017), 1479-1493.

\bibitem{De} C. Demeter, Fourier restriction, decoupling, and Applications, Cambridge studies in advanced mathematics, 184. Cambridge University press, 2020.

%

\bibitem{Ford} G. A. Ford, The fundamental solution and
Strichartz estimates for the Schr\"odinger equation on flat
Euclidean cones, Comm. Math. Phys., 299(2010), 447-467.

\bibitem{GS}  G. Gigante and F. Soria, On the boundedness in $H^{1/4}$ of the maximal square function
associated with the Schr\"odinger equation, J. London
Math. Soci., 77(2008), 51-68.


\bibitem{G} L. Grafakos, Classical Fourier Analysis, Second Edition, Graduate Texts in Math., no. 249,
Springer, New York, 2008.


\bibitem{GHS} C. Guillarmou, A. Hassell and A. Sikora, Restriction and
spectral multiplier theorems on asymptotically conic manifolds,
Analysis and PDE, 6(2013), 893-950.

\bibitem{Guth1} L. Guth, A restriction estimate using polynomial partitioning, J. Amer. Math. Soc. 29 (2016), no. 2,
371–413.

\bibitem{Guth2}  L. Guth, Restriction estimates using polynomial partitioning II, Acta Math., 221(2018), 81-142.

%
%
%


\bibitem{HL} A. Hassell and P. Lin,  The Riesz transform for homogeneous Schr\"odinger operators on metric cones.
Rev. Mat.Iberoamericana 30,477-522(2014)

\bibitem{HTW} A. Hassell, T. Tao and J. Wunsch, Sharp Strichartz estimates on
non-trapping asymptotically conic manifolds, Amer. J. Math.,
128(2006), 963-1024.

\bibitem{HZ} A. Hassell and J. Zhang, Global-in-time Strichartz estimates on nontrapping asymptotically conic manifolds,
Analysis \& PDE, 9(2016), 151-192.

\bibitem{KSS}
M. Keel, H. Smith and C. D. Sogge,
 Almost global existence for some semilinear wave equations, Dedicated to the memory of Thomas  Wolff.
 J. Anal. Math. 87, 265-279 (2002).
%
%
%




\bibitem{Li3} H. Li, $L^p-$estimates for the wave equation on manifolds with conical singularities. Math. Z., 272(2012), 551-575.


\bibitem{Li1} H. Li, La transformation de Riesz sur les vari$\acute{e}$t$\acute{e}$s coniques.J.Funct.Anal., 168(1999), 145-238.


\bibitem{Li2} H. Li, Estimations du noyau de la chaleur sur les vari$\acute{e}$t$\acute{e}$s coniques et ses applications. Bull. Sci. Math.,
124(2000), 365-384.

\bibitem{MS} W. P. Minicozzi II, C. D. Sogge, Negative results for Nikodym maximal functions and related oscillatory integrals in curved space, Math. Res. Lett., 4(1997), 221-237.

%

\bibitem{MZZ} C. Miao, J. Zhang, and J. Zheng,
 A note on the cone restriction conjecture, Proc. AMS, 140(2012), 2091-2102.




\bibitem{MZZ1} C. Miao, J. Zhang, and J. Zheng,
 Linear adjoint restriction estimates for paraboloid,  Math. Z. 292(2019), 427-451.

\bibitem{Moc} G. Mockenhaupt, On radial weights for the spherical summation operator, J. Funct. Anal., 91(1990), 174-181.

%


\bibitem{MSe} D. M\"uller and A. Seeger, Regularity properties of wave propagation on conic manifolds and applications to spectral multipliers,
Adv. Math.,  161(2001), 41-130.

\bibitem{MT} J. Metcalfe and D. Tataru, Global parametrices and dispersive
estimates for variable coefficient wave equations. Math. Ann.
353(2012), 1183-1237.


\bibitem{Mooer} E. Mooer, Heat kernel asymptotics on manifolds with conic singularities. J.Anal.Math., 78(1999), 1-36.

\bibitem{OW} Y. Ou and H. Wang, A cone restriction estimate using polynomial partitioning, arXiv:1704.05485.

\bibitem{PSS} F. Planchon, J. Stalker and A. S. Tahvildar-Zadeh, $L^p$ estimates for the wave equation with the inverse-square potential.
Discrete Contin. Dynam. Systems, 9(2003), 427-442.

\bibitem{RS}  I. Rodnianski and W. Schlag, Time decay for solutions of Schr\"odinger equations with rough and time-dependent potentials, Invent. Math., 155(2004):451-513.



\bibitem{Shao} S. Shao, A note on the cone restriction conjecture in the cylindrically symmetric case, Proc.
Amer. Math. Soc. 137(2009),135-143.


\bibitem{ST} G. Staffilani and D. Tataru, Strichartz estimates for a Schr\"odinger
operator with nonsmooth coefficients, Comm. Part. Diff. Eq.,
27(2002), 1337-1372.


%

\bibitem{sogge} C. D. Sogge, Fourier integrals in classical analysis, Cambridge Tracts in Mathematics 105, Cambridge University Press, 1993.

\bibitem{Stein} E. M. Stein, Some problems in harmonic analysis, Harmonic
analysis in Euclidean spaces (Proc. Sympos. Pure Math., Williams
Coll. Williamstown, Mass., 1978), Part1, pp. 3-20.



\bibitem{Stein79} E. Stein, Some problems in harmonic analysis, Harmonic analysis in Euclidean spaces, 1979, 3-20.


%


\bibitem{SW} E.M. Stein and G. Weiss, Introduction to Fourier analysis on Euclidean spaces,
Princeton University Press, Princeton, N. J., 1971, Princeton
Mathematical Series, No. 32. MR0304972.

%

%
%


\bibitem{Tao} T. Tao (2004)
Some recent progress on the restriction conjecture. In: Brandolini L., Colzani L., Travaglini G., Iosevich A. (eds) Fourier Analysis and Convexity. Applied and Numerical Harmonic Analysis. Birkh\"auser, Boston, MA.

\bibitem{Taylor} M. Taylor, (1996). Partial Differential Equations, Vol II. Berlin: Springer.
%
%
%
%


\bibitem{Wang} X. Wang, Asymptotic expansion in time of the Schr$\ddot{o}$dinger group on conical manifolds. Ann.Inst.Fourier, 56(2006), 1903-1945.


%

%


\bibitem{Watson} G. N. Watson,  A Treatise on the Theory of Bessel Functions. Second Edition
Cambridge University Press, (1944).

\bibitem{Wolff} T. Wolff,  A sharp bilinear cone restriction estimate, Ann. of Math.,
153(2001), 661-698.

%


\bibitem{Zhang} J. Zhang, Linear restriction estimates for Schr\"odinger equation on metric cones, Commu. in PDF,
40(2015),995-1028.

%


\bibitem{ZZ1} J. Zhang and J. Zheng, Global-in-time Strichartz estimates and cubic Schr\"{o}dinger equation in a conical singular space,arXiv:1702.05813


\bibitem{ZZ2} J. Zhang and J. Zheng, Strichartz estimates and wave equation in a conic singular space,  Math. Ann., 376(2020),525–581.




\end{thebibliography}
\end{document}